\documentclass[11pt,a4paper]{article}
\usepackage{epsfig}
\usepackage{graphics}
\usepackage{amsmath,amssymb,amsfonts}
\usepackage{subfigure}
\usepackage[mac]{inputenc}
\usepackage{graphicx}
\usepackage{epstopdf}
\usepackage{hyperref}
\usepackage{pdfsync}
\usepackage{color}
\usepackage{dsfont}

\newtheorem{ass}{Assumption}
\newtheorem{theo}{Theorem}
\newtheorem{proposition}{Proposition}
\newtheorem{lemma}{Lemma}

\newcommand \cqfd{$\qquad \square$}
\numberwithin{equation}{section}
\numberwithin{ass}{section}
\numberwithin{theo}{section} \numberwithin{proposition}{section}
\numberwithin{lemma}{section}
\newcommand{\re}{\Re\mathrm{e}} 
\def\R{\mathbb{ R}}

\def\Z{\mathbb{ Z}}

\newcommand{\PP}{\mathbb {P}}

\newcommand{\E}{\mathbb{E}}
\newcommand{\btau}{\boldsymbol{\tau}}

\newcommand{\N}{\mathbb{N}}

\newcommand{\var}{\text{Var}}

\newcommand{\opO}{\ensuremath{\mathcal O}}
\newcommand{\thefont}[2]{\fontsize{#1}{#2}\fontshape{n}\selectfont}
\newcommand{\1}{\rlap{\thefont{10pt}{12pt}1}\kern.16em\rlap{\thefont{11pt}{13.2pt}1}\kern.4em}

\title{Intensity estimation of non-homogeneous Poisson processes from shifted trajectories}

\author{Jérémie Bigot, Sébastien Gadat, Thierry Klein and Clément Marteau \vspace{0.2cm}  \\
Institut de Math\'ematiques de Toulouse\\
Universit\'e de Toulouse et CNRS (UMR 5219)\\
31062 Toulouse, Cedex 9, France\\
{\small {\tt \{Jeremie.Bigot, Sebastien.Gadat, Thierry.Klein, Clement.Marteau\}@math.univ-toulouse.fr }}
 \vspace{0.2cm}}

\begin{document}
\maketitle

\begin{abstract}
This paper considers the problem of adaptive estimation of a non-homogeneous intensity function from the observation of $n$ independent Poisson processes having a common intensity that is randomly shifted for each observed trajectory. We show that estimating this intensity  is a deconvolution problem for which the density of the random shifts plays the role of the convolution operator. In an asymptotic setting where the number $n$ of observed trajectories tends to infinity, we derive upper and lower bounds for the minimax quadratic risk over Besov balls. Non-linear thresholding in a Meyer wavelet basis is used to derive an adaptive estimator of the   intensity. The proposed estimator is shown to achieve a near-minimax rate of convergence.  This rate depends both on the smoothness of the intensity function and the density of the random shifts, which makes a connection between the classical deconvolution problem in nonparametric statistics and the estimation of a mean intensity from the observations of independent Poisson processes.
\end{abstract}



\subsubsection*{Acknowledgements}

The authors acknowledge the support of the French Agence Nationale de la Recherche (ANR) under reference ANR-JCJC-SIMI1 DEMOS.


\section{Introduction}

Poisson processes became intensively studied in the statistical theory during the last decades. Such processes are well suited to model a large amount of phenomena. In particular, they are used in various applied fields including genomics, biology and imaging. In this paper, we consider the problem of estimating nonparametrically a mean pattern intensity $\lambda$ from the observation of $n$ independent and non-homogeneous Poisson processes $N^1,\dots,N^n$ on the interval $[0,1]$. This problem arises when data (counts) are collected independently from $n$ individuals according to similar Poisson processes. In many applications, such data can be modeled as independent Poisson processes whose non-homogeneous intensities have a common shape. A simple model, that is well studied for genomics applications  \cite{superman}, is to assume that the intensity functions $\lambda_{1},\ldots,\lambda_{n}$ of the Poisson processes $N^1,\dots,N^n$ are randomly shifted versions $\lambda_{i}(\cdot) = \lambda(\cdot - \btau_{i})$ of a common intensity $\lambda$, where $\btau_{1},\ldots,\btau_{n}$ are i.i.d.\ random variables. The intensity $\lambda$ that we want to estimate is thus the same for all the observed processes up to random translations.  Basically, such a model corresponds to the assumption that the recording of counts does not start at the same time (or location) from one individual to another, e.g.\ when reading DNA sequences from different subjects in genomics   \cite{bioman}. 


In more rigorous terms, let  $\btau_{1},\ldots,\btau_{n}$ be i.i.d.\ random variables  with known density $g$ with respect to the Lebesgue measure on $\R$. Let $\lambda : [0,1] \to \R_{+}$ a real-valued function. Throughout the paper, it is assumed that $\lambda$  can be extended outside $[0,1]$ by periodization i.e.\ by taking $\lambda(t) = \lambda(t \mod 1)$ for $t \notin [0,1]$, where $t \mod 1$ denotes the modulo operation. We suppose that, conditionally to $\btau_{1},\ldots,\btau_{n}$, the point processes $N^1,\dots,N^n$ are independent Poisson processes on the measure space $([0,1],\mathcal{B}([0,1]),dt)$ with intensities $\lambda_{i}(t) = \lambda(t - \btau_{i})$ for $t \in [0,1]$, where $dt$ is the Lebesgue measure. Hence, conditionally to $\btau_{i}$, $N^i$ is a random countable set of points in $[0,1]$, and we denote by $d N^{i}_{t} = d N^{i}(t)$ the discrete random measure $ \sum_{T \in N^{i}}  \delta_{T}(t)$ for $t \in [0,1]$, where $\delta_{T}$ is the Dirac measure at point $T$. In other terms, conditionally to $\btau_{1},\ldots,\btau_{n}$, one has that for any  set $A \in \mathcal{B}([0,1])$ and for each $1 \leq i \leq n$, the number of points of $N^i$ lying in $A$ is a random variable $N^i(A) = \int_A d N^i_{t} = \int_A d N^i(t) $ which is Poisson distributed with parameter $\int_A \lambda(t - \btau_{i}) dt$. Moreover, for all finite family of disjoint measurable sets $A_1,\dots, A_p$ of $\mathcal{B}([0,1])$, the random variables $N^i(A_1), \dots, N^i(A_p), i =1\ldots,n$ are independent. For an introduction to non-homogeneous Poisson processes we refer to \cite{poisson-procs}. The objective of this paper is to study the estimation of $\lambda$ from a minimax point of view as the number $n$ of observed Poisson processes tends to infinity.

Denote by $\| \lambda \|^{2}_{2} = \int_{0}^{1} | \lambda(t)|^{2}dt$ the squared norm of a function $ \lambda $ belonging to the space $L^{2}([0,1])$ of squared integrable functions on $[0,1]$ with respect to $dt$. Let $\Lambda \subset L^{2}([0,1])$ be some smoothness class of functions, and let $\hat{\lambda}_{n} \in L^{2}([0,1])$ denote an estimator of the intensity function $\lambda \in \Lambda$, i.e a measurable mapping of the random processes $N^{i}, \; i=1,\ldots,n$ taking its value in $L^{2}([0,1])$.   Define the quadratic risk of the estimator $\hat{\lambda}_{n}$ as
$$
\mathcal{R}(\hat{\lambda}_{n},\lambda) = \E \|\hat{\lambda}_{n} - \lambda\|^{2}_{2},
$$
and introduce the following minimax risk
$$
\mathcal{R}_{n}(\Lambda) = \inf_{\hat{\lambda}_{n}} \sup_{\lambda \in \Lambda} \mathcal{R}(\hat{\lambda}_{n}, \lambda),
$$
where the above infimum is taken over the set of all possible estimators constructed from $N^1,\dots,N^n$. In order to investigate the optimality of an estimator, the main contribution of this paper is to derive upper and lower bounds for $\mathcal{R}_{n}(\Lambda)$ when $\Lambda $ is a Besov ball, and the construction of an adaptive estimator that achieves a near-minimax rate of convergence.

The estimation of the intensity of non-homogeneous Poisson process  has recently attracted a lot of attention in nonparametric statistics. In particular the problem of estimating a Poisson intensity from a single trajectory has been studied using model selection techniques \cite{Reynaud} and non-linear wavelet thresholding \cite{MR1268002},  \cite{MR1678884}, \cite{Reynaud2}, \cite{MR2417680}. Adopting an inverse problem point of view, estimating the intensity function of an indirectly observed non-homogeneous Poisson process has been considered by   \cite{MR2291495}, \cite{MR1930348}, \cite{MR1790322}. Poisson noise removal has also been considered  by \cite{MR2649451}, \cite{MR2516691} for image processing applications. Deriving optimal estimators of a Poisson intensity using a minimax point of view has been considered in \cite{MR1930348},   \cite{Reynaud}, \cite{Reynaud2} \cite{MR2417680}, in the setting where the intensity of the observed process $\lambda(t) = \kappa \lambda_{0}(t)$ such that the function to estimate is the scaled intensity $\lambda_{0}$ and $\kappa$ is a positive real, representing an ``observation time'', that is let going to infinity to study asymptotic properties.

In this paper, since we observe $n$ independent Poisson processes, we adopt a different asymptotic setting where $n$ tends to infinity. In this framework, our main result is that estimating $\lambda$ corresponds to a deconvolution problem where the density $g$ of the random shifts $\btau_{1},\ldots,\btau_{n}$ is a convolution operator that has to be inverted. Hence, estimating $\lambda$ falls into the category of Poisson inverse problems. A related model of randomly shifted curves observed with Gaussian noise has been considered by \cite{MR2676894} and \cite{MR2727451}. The results in \cite{MR2676894} show that  estimating a mean shape curve in such models is a deconvolution problem. However, to the best of our knowledge, the case of estimating a mean intensity from randomly shifted trajectories in the case of a Poisson noise has not been considered before. The presence of the random shifts significantly complicates the construction of upper and lower bounds for the minimax risk. In particular, to derive a lower bound, standard methods such as Assouad's cube technique  that is widely used for standard deconvolution problems in a white noise model (see e.g.\ \cite{MR2488345} and references therein) have to be carefully adapted to  take into account the effect of the random shifts.

The rest of the paper is organized as follows. In Section \ref{sec:nonadapt}, we describe an inverse problem formulation for the estimation of $\lambda$, and a linear but nonadaptive estimator of the intensity is proposed. Section \ref{sec:adapt} is devoted to adaptive estimation using non-linear Meyer wavelet thresholding, and to the construction of an upper bound on the minimax risk over Besov balls. In Section \ref{sec:lowerbound} a lower bound on the minimax risk is derived. 


\section{Linear estimation} \label{sec:nonadapt}

\subsection{Inverse problem formulation}
 For each observed counting process,   the presence of a random shift complicates the estimation of the intensity $\lambda$. Indeed, for all $i \in \lbrace 1,\dots, n \rbrace$ and any $f \in L^{2}([0,1])$ we have
\begin{equation}
\mathbb{E} \left[ \int_0^1 f(t)dN_t^i \big| \btau_i \right]= \int_0^1 f(t)\lambda(t-\btau_i)dt,
\label{eq:eps}
\end{equation}
where $\mathbb{E}[.|\btau_i]$ denotes the conditionnal expectation with respect to the variable $\btau_i$. Thus
$$ \mathbb{E} \int_0^1 f(t)dN_t^i = \int_0^1 f(t) \int_{\R} \lambda(t-\tau)g(\tau)d\tau dt = \int_0^1 f(t) (\lambda\star g)(t) dt.$$
Hence, the mean intensity of each randomly shifted process is the convolution $\lambda\star g$ between $\lambda$ and the density of the shifts $g$.   This shows that  a parallel can be made with the classical statistical deconvolution problem which is known to be an inverse problem. This parallel is highlighted by taking a Fourier transformation of the data. Let $(e_{\ell})_{\ell\in \Z}$ the complex Fourier basis on $[0,1]$, i.e. $e_{\ell}(t)=e^{i2\pi \ell t}$ for all $\ell\in \Z$ and $t\in [0,1]$. For $\ell\in \Z$, define
$$ \theta_{\ell}= \int_0^1 \lambda(t) e_{\ell}(t) dt \ \mathrm{and} \ \gamma_{\ell} = \int_0^1 g(t) e_{\ell}(t) dt,$$
as the Fourier coefficients of the intensity $\lambda$ and the density $g$ of the shifts. Then, for $\ell\in \Z$, define $y_{\ell}$ as
\begin{equation} 
y_{\ell} = \frac{1}{n} \sum_{i=1}^n \int_{0}^1 e_{\ell}(t) dN^i_t.
\label{eq:y_k}
\end{equation}
Using (\ref{eq:eps}) with $f=e_{\ell}$, we obtain that 
$$ \mathbb{E} \left[ y_{\ell} \big| \btau_1,\dots,\btau_n \right] = \frac{1}{n} \sum_{i=1}^n  \int_{0}^1 e_{\ell}(t) \lambda(t-\btau_i) dt = \frac{1}{n} \sum_{i=1}^n e^{-i 2 \pi \ell \btau_i} \theta_{\ell} = \tilde\gamma_{\ell} \theta_{\ell},$$
where we have used the notation 
\begin{equation} \label{eq:gammatilde}
\tilde \gamma_{\ell} = \frac{1}{n} \sum_{i=1}^n e^{i 2 \pi \ell \btau_i}, \ \forall \ell \in \Z.
\end{equation}
Hence, the estimation of the intensity $\lambda$ can be formulated as follows: we want to estimate the sequence $(\theta_{\ell})_{\ell \in\Z}$ of Fourier coefficients of $\lambda$ from the sequence space model
\begin{equation}
y_{\ell} = \tilde \gamma_{\ell} \theta_{\ell} + \xi_{\ell,n},
\label{eq:ssm}
\end{equation}
where the $\xi_{\ell,n}$ are centered random variables defined as
$$ \xi_{\ell,n}=\frac{1}{n} \sum_{i=1}^n \left[ \int_{0}^1 e_{\ell}(t) dN_t^{i} - \int_{0}^1 e_{\ell}(t)  \lambda(t-\btau_i) d t \right]  \mbox{ for all } \ell\in \Z. $$
The model (\ref{eq:ssm}) is very close to the standard formulation of statistical linear inverse problems. Indeed, using the singular value decomposition of the considered operator, the standard sequence space model of an ill-posed statistical inverse problem is (see \cite{cavgopitsy} and the references therein)
\begin{equation} \label{eqinvpb}
c_{\ell} = \theta_{\ell} \gamma_{\ell} +  z_{\ell},
\end{equation}
where the $\gamma_{\ell}$'s are eigenvalues of a known linear operator, and the $z_{\ell}$'s represent an additive random noise. The issue in model \eqref{eqinvpb} is to recover the coefficients $\theta_{\ell}$ from the observations $c_{\ell}$. A large class of estimators in model (\ref{eqinvpb}) can be written as
$$
\hat{\theta}_{\ell} = \delta_{\ell}  \frac{c_{\ell}}{\gamma_{\ell}},
$$
where $\delta = (\delta_{\ell})_{\ell \in \Z}$ is a sequence of reals with values in $[0,1]$ called filter (see \cite{cavgopitsy} for further details). 

Equation (\ref{eq:ssm}) can be viewed as a linear inverse problem with a Poisson noise for which the operator to invert is stochastic with eigenvalues $\tilde{\gamma}_{\ell}$ \eqref{eq:gammatilde} that are unobserved random variables. Nevertheless, since the density $g$ of the shifts is assumed to be known and given that $$\mathbb{E}\tilde{\gamma}_{\ell} =  \gamma_{\ell}$$ and  $\tilde{\gamma}_{\ell} \approx  \gamma_{\ell}$  for $n$ sufficiently large (in a sense which will be made precise later on), an estimation of the Fourier coefficients of $f$ can be obtained by a deconvolution step of the form
\begin{equation}
\hat{\theta}_{\ell} = \delta_{\ell} \frac{y_{\ell}}{\gamma_{\ell}} \label{eq:deftheta},
\end{equation}
where $\delta = (\delta_{\ell})_{\ell \in \Z}$ is a filter whose choice has to be discussed.

In this paper, the following type of assumption on  $g$ is considered:

\begin{ass}\label{ass:g}
The Fourier coefficients of $g$ have a polynomial decay {\it i.e.} for some real $\nu > 0$, there exist two constants $C\geq C' > 0$ such that
$
C' |\ell|^{-\nu} \leq  |\gamma_{\ell}| \leq C |\ell|^{-\nu}
$
for all $\ell \in \Z$.
\end{ass}

In standard inverse problems such as deconvolution, the expected optimal rate of convergence from an arbitrary estimator typically depends on such smoothness assumptions for $g$. The parameter $\nu$ is usually referred to as the degree of ill-posedness of the inverse problem, and it quantifies the difficult of inverting the convolution operator.   We will also need the following technical assumption on the decay of the density $g$, which is not a very restrictive condition as $g$ is supposed to be an integrable function on $\R$.

\begin{ass} \label{ass:gdecay}
There exists a constant $C > 0$ and a real $\alpha > 1$  such that the density  $g$  satisfies
$
g(x) \leq \frac{C}{1+ |x|^{\alpha}} \mbox{ for all } x \in \R.
$
\end{ass}

\subsection{A linear estimator by spectral cut-off}

First, we propose a non-adaptive estimator in order to derive an upper bound on the minimax risk. This part allows us to shed light on the connexion between our model and a deconvolution problem. For a given filter $(\delta_{\ell})_{\ell\in\Z}$ and using \eqref{eq:deftheta}, a linear estimator of $\lambda$ is given by
\begin{equation}
\hat\lambda_{\delta}(t)= \sum_{\ell\in \Z} \hat{\theta}_{\ell} e_{\ell}(t)  = \sum_{\ell\in \Z} \delta_{\ell} \gamma_{\ell}^{-1} y_{\ell} e_{\ell}(t)  , \; t \in [0,1], \label{eq:lambdalin}  
\end{equation}
whose quadratic risk can be written in the Fourier domain as
$$ \mathcal{R}(\hat\lambda_{\delta},\lambda)  = \mathbb{E} \sum_{\ell\in \Z} (\hat\theta_{\ell} - \theta_{\ell})^2.$$
The following proposition illustrates how the quality of the estimator $\hat\lambda_{\delta}$ (in term of quadratic risk) is related to the choice of the  filter $\delta$.  
\begin{proposition} \label{prop:risklin}
For any given non-random filter $\delta$, the risk of $\hat{\lambda}^{\delta}$ can be decomposed as
\begin{equation}
\mathcal{R}(\hat{\lambda}^{\delta},\lambda) = \sum_{\ell \in \mathbb{Z} } |\theta_{\ell}|^2 (\delta_{\ell}-1)^2 +  \sum_{\ell\in \mathbb{Z}}  \frac{\delta_{\ell}^2}{n}  |\gamma_{\ell}|^{-2} \|\lambda\|_{1}   + \sum_{\ell\in \mathbb{Z}}  \frac{\delta_{\ell}^2}{n}   |\theta_{\ell}|^2 \left( |\gamma_{\ell}|^{-2} -1  \right) .
\label{eq:risque}
\end{equation}
where $\|\lambda\|_{1} = \int_{0}^{1} \lambda(t) dt$.
\end{proposition}

\noindent
\textsc{Proof}. Remark that
\begin{eqnarray}\label{eq:diff}
\hat{\theta}_{\ell} - \theta_{\ell}& = &\theta_{\ell} \left[\delta_{\ell} \frac{\tilde{\gamma}_{\ell}}{\gamma_{\ell}}-1\right] + \frac{\delta_{\ell}}{n}\sum_{i=1}^n\epsilon_{\ell,i},
\end{eqnarray}
where the $\epsilon_{\ell,i}$ are centered random variables defined as
$
\epsilon_{\ell,i}=\gamma_{\ell}^{-1}\int_{0}^1 e_{\ell}(t) \left( dN_t^{i} - \lambda(t-\btau_i)dt \right).
$
Now, to compute $ \mathbb{E} |\hat{\theta}_{\ell}-\theta_{\ell}|^2 $, remark first that
\begin{eqnarray*}
|\hat{\theta}_{\ell}-\theta_{\ell}|^2 &= &\left[ \theta_{\ell} \left[\delta_{\ell} \frac{\tilde{\gamma}_{\ell}}{\gamma_{\ell}}-1\right] + \frac{\delta_{\ell}}{n} \sum_{i=1}^n \epsilon_{\ell,i} \right]
\overline{
\left[ \theta_{\ell} \left[\delta_{\ell}\frac{\tilde{\gamma}_{\ell}}{\gamma_{\ell}}-1\right] + \frac{\delta_{\ell}}{n} \sum_{i=1}^n \epsilon_{\ell,i} \right]}\\
& = & \left[ |\theta_{\ell}|^2   \left| \delta_{\ell} \frac{\tilde{\gamma}_{\ell}}{\gamma_{\ell}}-1\right|^2
+ 2 \re\left(\theta_{\ell} \left[ \delta_{\ell} \frac{\tilde{\gamma}_{\ell}}{\gamma_{\ell}}-1\right] \overline{\frac{\delta_{\ell}}{n} \sum_{i=1}^n \epsilon_{\ell,i} }\right)
+ \frac{\delta_{\ell}^2}{n^2} \sum_{i,i'=1}^n \epsilon_{\ell,i} \overline{\epsilon_{\ell,i'}} \right]. \\
\end{eqnarray*}
Taking expectation in the above expression yields
\begin{eqnarray*}
\mathbb{E} |\hat{\theta}_{\ell}-\theta_{\ell}|^2 & = & \mathbb{E} \left[ \mathbb{E}  |\hat{\theta}_{\ell}-\theta_{\ell}|^2  \big| \btau_{1}, \ldots, \btau_{n} \right]\\
& = &  \mathbb{E} \left[ |\theta_{\ell}|^2 \left| \delta_{\ell} \frac{\tilde{\gamma}_{\ell}}{\gamma_{\ell}}-1\right|^2 +
2 \re\left(\theta_{\ell} \left[\delta_{\ell} \frac{\tilde{\gamma}_{\ell}}{\gamma_{\ell}}-1\right] \overline{\mathbb{E} \left[ \frac{\delta_{\ell}}{n} \sum_{i=1}^n \epsilon_{\ell,i} \right] } \right)    \big| \btau_{1}, \ldots, \btau_{n}  \right] \\
&  & + \mathbb{E} \left[  \frac{\delta_{\ell}^2}{n^2} \sum_{i,i'=1}^n \mathbb{E}\left[\epsilon_{\ell,i} \overline{\epsilon_{\ell,i'}}   \big| \btau_{1}, \ldots, \btau_{n}  \right]\right].
\end{eqnarray*}
Now, remark that given two integers $i \neq i'$ and the two shifts $\btau_{i}, \btau_{i'}$, $\epsilon_{\ell,i}$ and $\epsilon_{\ell,i'}$ are independent with zero mean. Therefore,   using the equality
$$ \mathbb{E} \left| \delta_{\ell} \frac{\tilde{\gamma}_{\ell}}{\gamma_{\ell}}-1\right|^2 = \delta_{\ell}^2 |\gamma_\ell|^{-2} \mathbb{E} |\tilde\gamma_\ell-\gamma_\ell|^2 + (\delta_\ell-1)^2 = (\delta_\ell-1)^2 + \frac{\delta_\ell^2}{n} (|\gamma_\ell|^{-2} -1),$$
one finally obtains
\begin{eqnarray*}
\mathbb{E} |\hat{\theta}_{\ell}-\theta_{\ell}|^2 & = & |\theta_{\ell}|^2 \mathbb{E} \left| \delta_{\ell} \frac{\tilde{\gamma}_{\ell}}{\gamma_{\ell}}-1\right|^2 + \mathbb{E} \left[   \frac{\delta_{\ell}^2}{n^2} \sum_{i=1}^n \mathbb{E} \left[|\epsilon_{\ell,i}|^2   \big| \btau_{1}, \ldots, \btau_{n}  \right]\right] \\
& = & |\theta_{\ell}|^2 (\delta_{\ell}-1)^2 +   \frac{\delta_{\ell}^2}{n} \left(   |\theta_{\ell}|^2  \left( |\gamma_{\ell}|^{-2}  -1 \right) + \mathbb{E} |\epsilon_{\ell,1}|^2 \right).
\end{eqnarray*}
Using in what follows the equality $\mathbb{E}|a+ib|^2 = \mathbb{E}[|a|^2+|b|^2]$
with
$
a=\int_{0}^1 \cos (2 \pi \ell t) \left( dN_t^1 - \lambda(t-\btau_1) dt \right)
$
and
$
b=\int_{0}^1 \sin (2 \pi \ell t) \left( dN_t^1 - \lambda(t-\btau_1) dt \right),
$
we obtain
\begin{eqnarray*}
\mathbb{E} |\epsilon_{\ell,1}|^2 &=& |\gamma_{\ell}|^{-2}\mathbb{E} \left[ \mathbb{E}  \left| \int_{0}^1 e_{\ell}(t) \left( dN_t^1 - \lambda(t-\btau_1)  dt \right) \right|^2  \big| \btau_{1} \right]\\
& = & |\gamma_{\ell}|^{-2} \mathbb{E}  \int_{0}^1 \left( | \cos (2 \pi \ell t) |^2 +  | \sin (2 \pi \ell t) |^2  \right) \lambda(t-\btau_{1}) dt = |\gamma_{\ell}|^{-2} \|\lambda\|_{1},
\end{eqnarray*}
where the last equality follows from the fact that $\lambda$ has been extended outside $[0,1]$ by periodization, which completes the proof.
\hfill $\Box$ \\

Note that the quadratic risk of any linear estimator in  model (\ref{eq:ssm}) is  composed of three  terms. The two first terms in the risk decomposition (\ref{eq:risque}) correspond to the classical bias and variance in statistical inverse problems. The third term corresponds to the error related to the fact that the inversion of the operator is done using $(\gamma_l)_{l\in\mathbb{Z}}$ instead of the (unobserved) random eigenvalues $(\tilde \gamma_l)_{l\in\mathbb{Z}}$.

\subsection{Upper bound of the minimax risk on Sobolev balls}

There exist different type of filters in the inverse problems literature (see e.g.\  \cite{cavgopitsy}). In this section, we consider the family of projection (or spectral cut-off) filters $\delta^{M} = \left( \delta_{\ell} \right)_{\ell \in \Z}=   \left(\1_{\lbrace |\ell|\leq M \rbrace} \right)_{\ell \in \Z}$ for some $M \in\N$. Using Proposition \ref{prop:risklin}, it follows that 
\begin{equation}
\mathcal{R}(\hat{\lambda}^{\delta^M},\lambda) = \sum_{\ell > M} |\theta_{\ell}|^2+ \frac{1}{n} \sum_{|\ell|<M}\left(  |\gamma_{\ell}|^{-2} \|\lambda\|_{1} + |\theta_{\ell}|^2 \left( |\gamma_{\ell}|^{-2} -1  \right)  \right) . \label{eq:risklinspec}
\end{equation}
Now, consider the following smoothness class of functions (a Sobolev ball of radius $A$)
$$
H_{s}(A) = \left\{\lambda \in L^{2}([0,1]) \; ; \sum_{\ell \in \Z}  (1+|\ell|^{2s}) |\theta_{\ell}|^{2} \leq A  \mbox{ and } \lambda(t) \geq 0    \mbox{ for all } t \in [0,1] \right\},
$$
and some smoothness parameter $s > 0$, where $\theta_{\ell} = \int_{0}^1 e^{- 2 i \ell \pi t} \lambda(t) dt$. For an appropriate choice of the spectral cut-off parameter $M$, the following proposition gives the asymptotic behavior of the risk of $\hat{\lambda}^{\delta^M}$, see equation \eqref{eq:lambdalin}.

\begin{proposition}\label{prop:risklinearpoly}
Assume that  $f$ belongs to $H_{s}(A)$ with $s > 1/2$ and $A > 0$, and  that $g$ satisfies Assumption (\ref{ass:g}). If $M = M_{n}$ is chosen as the largest integer such  $M_{n} \leq n^{\frac{1}{2s+ 2\nu+1}}$, then as $n \to + \infty$
$$
\sup_{\lambda \in H_{s}(A) } \mathcal{R}(\hat{\lambda}^{\delta^M},\lambda) = \opO \left(   n^{-\frac{2s}{2s + 2 \nu+1}} \right).
$$
\end{proposition}
The proof follows immediately from the decomposition \eqref{eq:risklinspec}, the definition of $H_{s}(A)$ and Assumption (\ref{ass:g}).

Hence, Proposition \ref{prop:risklinearpoly} shows that under Assumption \ref{ass:g} the quadratic risk $\mathcal{R}(\hat{\lambda}^{\delta^M},\lambda)$ is of polynomial order of the sample size $n$, and that this rate deteriorates as the degree of ill-posedness $\nu$ increases. Such a behavior is a well known fact for standard deconvolution problems, see e.g.\  \cite{MR2488345}, \cite{JKPR}   and references therein.  Proposition \ref{prop:risklinearpoly} shows that a similar phenomenon holds for our linear estimator. Hence, there may exist a connection between  estimating a mean pattern intensity from a set of non-homogeneous Poisson processes and the statistical analysis of deconvolution problems. However, the choice of $M_{n}$ depends on the a priori unknown smoothness $s$ of the intensity $\lambda$. Such a spectral cut-off estimator is thus non-adaptive and is of limited interest for applications. Moreover, the result of Proposition \ref{prop:risklinearpoly}  is only suited for smooth functions since Sobolev balls  $H_{s}(A)$ for $s > 1/2$ are not well adapted to model intensities $\lambda$ which may have singularities. In the following section, we thus consider the problem of constructing an adaptive estimator and we study its minimax risk over Besov balls.

\section{Adaptive estimation in Besov spaces} \label{sec:adapt}

\subsection{Meyer wavelets}

We will  use Meyer wavelets to obtain   a non-linear and adaptive estimator. Let us denote by $\psi$ (resp. $\phi$) the periodic mother Meyer wavelet (resp. scaling function) on the interval $[0,1]$ (see e.g. \cite{MR2488345,JKPR} for a precise definition). The intensity  $\lambda \in L^2([0,1])$ can then be decomposed as follows
$$
\lambda(t) = \sum_{k=0}^{2^{j_0}-1} c_{j_0,k} \phi_{j_0,k}(t) + \sum_{j=j_0}^{+ \infty} \sum_{k=0}^{2^j-1} \beta_{j,k} \psi_{j,k}(t),
$$ 
where $\phi_{j_0,k}(t)  = 2^{j_{0}} \phi(2^{j_{0}}t - k)$, $\psi_{j,k}(t)  = 2^{j} \psi(2^{j}t - k)$, $j_0 \geq 0 $ denotes the usual coarse level of resolution, and
$$
c_{j_0,k} = \int_{0}^{1} \lambda(t) \phi_{j_0,k}(t) dt, \; \beta_{j,k} = \int_{0}^{1} \lambda(t) \psi_{j,k}(t) dt,
$$
are the scaling and wavelet coefficients of $\lambda$. It is well known that Besov spaces can be characterized in terms of wavelet coefficients (see e.g \cite{Meyer}). Let $s > 0$ denote the usual smoothness parameter, then  for the Meyer wavelet basis and for a Besov ball $B^{s}_{p,q}(A)$ of radius $A > 0$ with $1 \leq p,q \leq \infty$, one has that
$$
B^{s}_{p,q}(A) = \left\{ f  \in L^{2}([0,1]) :  \left(\sum_{k =0}^{2^{j_{0}}-1} |c_{j_{0},k}|^{p} \right)^{\frac{1}{p}} + \left( \sum_{j = j_{0}}^{+ \infty} 2^{j(s + 1/2-1/p)q} \left( \sum_{k=0}^{2^{j}-1} |\beta_{j,k}|^{p}\right)^{\frac{q}{p}} \right)^{\frac{1}{q}} \leq A \right\}
$$ 
with the respective above sums replaced by maximum if $p=\infty$ or $q=\infty$. The parameter $s$ is related to the smoothness of the function $f$. Note that if $p=q=2$, then a Besov ball is equivalent to a Sobolev ball if $s$ is not an integer. For $1 \leq p < 2$,  the space $B^{s}_{p,q}(A) $ contains functions with local irregularities.

Meyer wavelets satisfy the fundamental property of being band-limited function in the Fourier domain which make them well suited for deconvolution problems. More precisely, each $\phi_{j,k}$ and $\psi_{j,k}$ has a compact support  in the Fourier domain in the sense that
$$
\phi_{j_{0},k}=\sum_{\ell \in D_{j_{0}}} c_\ell(\psi_{j_{0},k}) e_\ell, \; \psi_{j,k}=\sum_{\ell \in \Omega_j} c_\ell(\psi_{j,k}) e_\ell,
$$
with
$$
c_\ell(\phi_{j_{0},k}) =  \int_{0}^1 e^{- 2 i \ell \pi t} \phi_{j_{0},k}(t) dt, \; c_\ell(\psi_{j,k}) =  \int_{0}^1 e^{- 2 i \ell \pi t} \psi_{j,k}(t) dt,
$$
and where $D_{j_{0}}$ and $\Omega_j$ are finite subsets of integers  such that $\# D_{j_{0}} \leq C 2^{j_{0}}$, $\# \Omega_{j} \leq C 2^{j}$ for some constant $C > 0$ independent of $j$ and
\begin{equation}
\Omega_j \subset [-2^{j+2}c_{0},-2^{j}c_{0}] \cup [2^{j}c_{0},2^{j+2}c_{0}] \label{eq:Omegaj}
\end{equation}
with $c_{0} = 2\pi/3$. Then, thanks to Parseval's relation
$
c_{j_{0},k} = \sum_{\ell \in D_{j_{0}}} c_{\ell}(\phi_{j_{0},k}) \theta_\ell , \; \beta_{j,k}  = \sum_{\ell \in \Omega_j} c_{\ell}(\psi_{j,k}) \theta_\ell.
$
and from the unfiltered estimator $\hat{\theta}_\ell = \gamma_{\ell}^{-1} y_\ell$ of each $\theta_\ell$,  see equation \eqref{eq:ssm},  one can build estimators of the scaling and wavelet coefficients by defining
\begin{equation}\label{eq:estimation}
\quad \hat{c}_{j_0,k} = \sum_{\ell \in \Omega_{j_0}} c_\ell(\psi_{j_0,k}) \hat{\theta}_\ell  \quad \text{and} \quad  \hat{\beta}_{j,k} = \sum_{\ell \in \Omega_j} c_\ell(\psi_{j,k}) \hat{\theta}_\ell .
\end{equation}

\subsection{Hard thresholding estimation}

We propose to use a non-linear hard thresholding estimator defined by
\begin{equation}\label{eq:threshold_estimate}
\hat{\lambda}_n^h = \sum_{k=0}^{2^{j_0(n)}-1} \hat{c}_{j_0,k} \phi_{j_0,k} + \sum_{j=j_0(n)}^{j_1(n)} \sum_{k=0}^{2^j-1} \hat{\beta}_{j,k} \1_{\{| \hat{\beta}_{j,k}  | \geqslant \hat{s}_{j}(n)\}}  \psi_{j,k}.
\end{equation}
In the above formula, $\hat{s}_{j}(n)$ refers to possibly random thresholds that depend on the resolution $j$, while $j_{0} = j_0(n)$ and $j_{1} = j_1(n)$ are the usual coarsest and highest resolution levels whose dependency on $n$ will be specified later on. Then, let us  introduce some notations. For all $j\in \mathbb{N}$, define
\begin{equation}
\sigma^{2}_{j} = 2^{-j}  \sum_{\ell \in \Omega_{j}} |\gamma_{\ell}|^{-2} \mbox{ and } \epsilon_{j} = 2^{-j/2}  \sum_{\ell \in \Omega_{j}} |\gamma_{\ell}|^{-1}, \label{eq:sigepsj}
\end{equation}
and for any $\gamma>0$, let
\begin{equation}
\tilde{K}_n (\gamma)=   \frac{1}{n}  \sum_{i=1}^{n} K_{i} + \frac{4 \gamma \log n }{3n} + \sqrt{\frac{2  \gamma \log n}{n^2}    \sum_{i=1}^{n} K_{i}  + \frac{5  \gamma^{2} (\log n )^{2}}{3n^2}  },
\label{eq:tildeK}
\end{equation}
where $K_{i} = \int_{0}^{1} dN^{i}_{t}$ is the number of points of the counting process $N^{i}$ for $i=1,\ldots,n$.
Introduce also the class of bounded intensity functions
$$
\Lambda_{\infty} = \left\{ \lambda \in L^{2}([0,1]); \;  \| \lambda \|_{\infty} < + \infty \mbox{ and } \lambda(t) \geq 0    \mbox{ for all } t \in [0,1] \right\},
$$
where $\| \lambda \|_{\infty} = \sup_{t \in [0,1]} \{ |\lambda(t)| \}$.


\begin{theo} \label{theo:upperbound} Suppose that $g$ satisfies Assumption \ref{ass:g} and Assumption \ref{ass:gdecay}. Let $1 \leq p \leq \infty$, $1 \leq q \leq \infty$ and $A > 0$. Let $ p'=\min(2,p)$, and assume that
$
s > 1/p'  \mbox{ and } (s + 1/2 -1/p')p > \nu(2-p).
$
Let $\delta > 0$ and suppose that the non-linear  estimator $ \hat{\lambda}_{n}^{h}$ \eqref{eq:threshold_estimate} is computed using the random thresholds
 
$$
\hat{s}_{j}(n) =  4 \left( \sqrt{ \sigma^{2}_{j}  \frac{ 2 \gamma \log n}{n}   \left( \|g\|_{\infty}   \tilde{K}_n(\gamma) + \delta \right) } +  \frac{\gamma\log n}{3 n}  \epsilon_{j} \right), \;  \mbox{ for } \; j_{0}(n) \leq j \leq j_{1}(n),
$$
 with $\gamma\geq2$, and where $ \sigma^{2}_{j}$ and  $\epsilon_{j}$ are defined in \eqref{eq:sigepsj}. Define $j_{0}(n)$ as the largest integer such that $2^{j_{0}(n)} \leq \log n $ and  $j_{1}(n)$ as the largest integer such that $2^{j_{1}(n)} \leq \left(\frac{n}{\log n}\right)^{\frac{1}{2 \nu +1}}$. 
Then, as $n \to + \infty$,
$$
\sup_{\lambda \in  B_{p,q}^{s}(A)\bigcap   \Lambda_{\infty} } \mathcal{R}( \hat{\lambda}_{n}^{h}, \lambda) = \opO \left(  \left(\frac{ \log n}{n}\right)^{\frac{2s}{2s + 2\nu +1} } \right).
$$
\end{theo}

Hence, Theorem \ref{theo:upperbound} shows that under Assumption \ref{ass:g} the quadratic risk of the non-linear estimator  $\hat{\lambda}_{n}^{h}$ is of polynomial order of the sample size $n$, and that this rate deteriorates as $\nu$ increases. Again, this result illustrates the connection between estimating a mean intensity  from the observation of Poisson processes and the analysis of inverse problems in nonparametric statistics.  Note that the choices of the random thresholds $\hat{s}_{j}(n)$ and the highest resolution level $j_{1}$ do not depend on the smoothness parameter $s$.  Hence, contrary to the linear estimator proposed in Section \ref{sec:nonadapt}, the non-linear estimator $\hat{\lambda}_{n}^{h}$ is said to be adaptive with respect to the unknown smoothness $s$. Moreover, the Besov spaces $B^{s}_{p,q}(A)$ may contain functions with local irregularities. The above described non-linear estimator is thus suitable for the estimation of non-globally smooth functions.

In Section \ref{sec:lowerbound}, we show that the rate $n ^{-\frac{2s}{2s + 2\nu +1}}$ is a lower bound for the asymptotic decay of the minimax risk over a large scale of Besov balls. Hence, the wavelet estimator that we propose is almost optimal up to a logarithmic term which is usually called the price to be paid for adaptivity.

\subsection{Proof of the upper bound}

Following standard arguments in wavelet thresholding to derive the rate of convergence of such non-linear estimators (see e.g.\  \cite{MR2488345}), one needs to bound the centered moment of order 2 and 4 of  $\hat{c}_{j_{0},k}$ and $\hat{\beta}_{j,k}$ (see Proposition \ref{prop:moments}), as well as the deviation in probability between $\hat{\beta}_{j,k}$ and $\beta_{j,k}$ (see Proposition \ref{prop:dev}).  In the proof, $C$, $C'$, $C_{1}$, $C_{2}$ denote positive constants that are independent of $\lambda$ and $n$, and whose value may change from line to line. The proof requires technical results that are postponed and proved in Section \ref{technicalupper}. First, define the following quantities
$$
\tilde{\psi}_{j,k}(t) = \sum_{\ell \in \Omega_j} \gamma_\ell^{-1} c_\ell(\psi_{j,k}) e_{\ell}(t),  \; \ V_j^2 = \|g\|_{\infty} 2^{-j} \sum_{\ell \in \Omega_j} \frac{|\theta_\ell|^2}{|\gamma_\ell|^2}, \; 
\delta_j = 2^{-j/2} \sum_{\ell \in \Omega_j} \frac{| \theta_\ell|}{|\gamma_\ell|},
$$
and  
\begin{equation}
\Delta^n_{jk}(\gamma) =   \sqrt{ \| \tilde{\psi}_{j,k} \|_2^2  \left(  \|g\|_{\infty}   \tilde{K}_n(\gamma) \frac{2 \gamma \log n}{n}     + u_{n}(\gamma)  \right) } +   \frac{\gamma\log n}{3 n}     \| \tilde{\psi}_{j,k}\|_{\infty},
\label{eq:deltajkn}
\end{equation}
where $\tilde K_n(\gamma)$ is introduced in (\ref{eq:tildeK}), $u_{n}(\gamma)$ is a sequence of reals such that  $u_{n}(\gamma) = o\left(  \frac{ \gamma \log n}{n} \right)$ as $n \to + \infty$.

\subsubsection{Proof of Theorem \ref{theo:upperbound}}

As classically done in wavelet thresholding, use the following risk decomposition
$$
\E \| \hat{\lambda}_{n}^{h} - \lambda \|^{2}_{2} = R_1 + R_2 + R_3 + R_4,
$$
where
\begin{eqnarray*}
R_1 & = & \sum_{k =0}^{2^{j_{0}}-1} \E (\hat{c}_{j_{0},k} \ - c_{j_{0},k})^{2}, \quad
R_2  =   \sum_{j = j_{0}}^{j_{1}}  \sum_{k=0}^{2^{j}-1}   \E \left[ ( \hat{\beta}_{j,k} - \beta_{j,k})^2 \1_{\{ |\hat{\beta}_{j,k}| \geq \hat{s}_{j}(n)  \}} \right] \\
R_3 & = &  \sum_{j = j_{0}}^{j_{1}}  \sum_{k=0}^{2^{j}-1}   \E \left[  \beta_{j,k}^{2}  \1_{\{ |\hat{\beta}_{j,k}| < \hat{s}_{j}(n)  \}} \right], \quad
R_4  =  \sum_{j = j_{1} + 1}^{+ \infty} \sum_{k =0}^{2^{j}-1} \beta_{j,k}^{2}. 
\end{eqnarray*}
\underline{Bound on $R_4$}: first, recall that following our assumptions, Lemma 19.1 of \cite{johnstone} implies that
\begin{equation}
\sum_{k = 0}^{2^{j}-1} \beta_{jk}^{2} \leq C 2^{-2js^{\ast}}, \mbox{ with }  s^{\ast} = s + 1/2 - 1/p', \label{eq:john}
\end{equation}
where $C$ is a constant depending only on $p,q,s,A$. Since by definition $2^{-j_{1}} \leq 2 (\frac{\log n}{n})^{-\frac{1}{2 \nu +1}}$, equation (\ref{eq:john})  implies that
$
R_4 = \opO \left( 2^{-2 j_{1} s^{\ast}}\right) = \opO \left(  (\frac{\log n}{n})^{-\frac{2s^{\ast}}{2 \nu +1}} \right). 
$
Note that in the case $p \geq 2$, then $s^{\ast} = s$ and thus $\frac{2s}{2\nu+1} > \frac{2s}{2s+2\nu+1}$. In the case $1 \leq p < 2$, then $s^{\ast} = s + 1/2-1/p$, and one can check that the conditions $s > 1/p$ and $s^{\ast} p > \nu (2-p)$ imply that $\frac{2 s^{\ast}}{2 \nu +1} >  \frac{2s}{2s+2\nu+1}$.
Hence in both cases one has that
\begin{equation}
R_4 = \opO \left(  n^{-\frac{2s}{2s+2 \nu +1}} \right). \label{eq:R4}
\end{equation}
\underline{Bound on $R_1$}: using Proposition \ref{prop:moments} and the inequality $2^{j_{0}} \leq  \log n $ it follows that
\begin{equation}
R_1 \leq C  \frac{2^{j_{0} (2\nu+1)}}{n} \leq C  \frac{ (\log n)^{2\nu+1}}{n} = \opO \left(   n^{ -\frac{2s}{2s+2 \nu +1}} \right).  \label{eq:R1}
\end{equation}

\noindent 
\underline{Bound on $R_2$ and $R_3$}.
remark that
$
R_2 \leq R_{21} + R_{22} \mbox{ and } R_3 \leq R_{31} + R_{32}
$
with
$$
R_{21} =  \sum_{j = j_{0}}^{j_{1}}  \sum_{k=0}^{2^{j}-1}   \E \left[ ( \hat{\beta}_{j,k} - \beta_{j,k})^2 \1_{\{ |\hat{\beta}_{j,k} - \beta_{j,k}| \geq \hat{s}_{j}(n)/2  \}} \right] , R_{22} =  \sum_{j = j_{0}}^{j_{1}}  \sum_{k=0}^{2^{j}-1} \E \left[ ( \hat{\beta}_{j,k} - \beta_{j,k})^2 \1_{\{ |\beta_{j,k}| \geq \hat{s}_{j}(n)/2  \}} \right],
$$
$$
R_{31} =   \sum_{j = j_{0}}^{j_{1}}  \sum_{k=0}^{2^{j}-1}   \E \left[  \beta_{j,k}^{2}  \1_{\{ |\hat{\beta}_{j,k} - \beta_{j,k}| \geq \hat{s}_{j}(n)/2  \}} \right]  \mbox{ and } R_{32} =   \sum_{j = j_{0}}^{j_{1}}  \sum_{k=0}^{2^{j}-1}  \E   \left[   \beta_{j,k}^{2}  \1_{\{ |\beta_{j,k}|  < \frac{3}{2} \hat{s}_{j}(n)  \}}  \right] . 
$$
Now, applying twice the Cauchy-Schwarz inequality, we get that
\begin{eqnarray*}
R_{21} + R_{31}  &=& \sum_{j = j_{0}}^{j_{1}}  \sum_{k=0}^{2^{j}-1}   \E \left[ \left( ( \hat{\beta}_{j,k} - \beta_{j,k})^2 +  \beta_{j,k}^{2}  \right) \1_{\{ |\hat{\beta}_{j,k} - \beta_{j,k}| \geq \hat{s}_{j}(n)/2  \}} \right] \\
& \leq &   \sum_{j = j_{0}}^{j_{1}}  \sum_{k=0}^{2^{j}-1}   \left( \left(\E  ( \hat{\beta}_{j,k} - \beta_{j,k})^{4} \right)^{1/2}+  \beta_{j,k}^{2}  \right) \left( \PP( |\hat{\beta}_{j,k} - \beta_{j,k}| \geq \hat{s}_{j}(n)/2  ) \right)^{1/2}
\end{eqnarray*}
\underline{Bound on $\PP( |\hat{\beta}_{j,k} - \beta_{j,k}| \geq \hat{s}_{j}(n)/2  )$}: using that $|c_{\ell}( \psi_{j,k} )| \leq 2^{-j/2} $ one has that $ \| \tilde{\psi}_{j,k} \|_2^2 \leq \sigma_{j}^2$ and   $ \| \tilde{\psi}_{j,k} \|_{\infty} \leq \epsilon_{j}$. Thus,
by definition of $\hat{s}_{j}(n)$   it follows that 
\begin{equation}
2 \Delta_{jk}^n(\gamma) \leq \hat{s}_{j}(n)/2    \label{eq:hats1}
\end{equation}
for all sufficiently large $n$ where $\Delta_{jk}^n(\gamma)$  is defined in (\ref{eq:deltajkn}).   Moreover, by \eqref{eq:Omegaj} there exists two constants $C_{1},C_{2}$ such that for all $\ell \in \Omega_{j}$, $C_{1} 2^{j}  \leq | \ell | \leq C_{2} 2^{j}$. Since $\lim_{|\ell| \to + \infty} \theta_{\ell} = 0$ uniformly for $f \in  B_{p,q}^{s}(A)$ it follows that as $j \to + \infty$
$$
V^{2}_{j} =   \| g \|_{\infty}  2^{-j}  \sum_{\ell \in \Omega_{j}} \frac{|\theta_{\ell}|^{2}}{ |\gamma_{\ell}|^{2}} = o\left( 2^{-j}  \sum_{\ell \in \Omega_{j}} |\gamma_{\ell}|^{-2} \right) = o\left( \sigma^{2}_{j} \right) \mbox{ and }  \delta_{j} =  2^{-j/2} \sum_{\ell \in \Omega_{j}}  \frac{|\theta_{\ell}|}{ |\gamma_{\ell}| }   = o \left( \epsilon_{j} \right).
$$
%
Now, define the non-random threshold
\begin{equation} \label{eq:thrfixed}
s_{j}(n) =   4 \left( \sqrt{ \sigma^{2}_{j}  \frac{2 \gamma \log n}{n}  \left(  \|g\|_{\infty}   \| \lambda \|_{1} + \delta \right)  } +  \frac{\gamma\log n}{3 n}  \epsilon_{j}\right) , \; \mbox{ for } \; j_{0}(n) \leq j \leq j_{1}(n).
\end{equation}
Using that $V^{2}_{j} = o( \sigma^{2}_{j} )$ and $\delta_{j}  = o \left( \epsilon_{j} \right)$ as $j \to + \infty$, and that $j_{0}(n) \to + \infty$ as $n \to + \infty$ it follows that for all sufficiently large $n$ and $j_{0}(n) \leq j \leq j_{1}(n)$
\begin{equation}
2 \left( \sqrt{\frac{2 V_j^2 \gamma \log n}{n}} + \delta_j \frac{\gamma \log n}{3n} \right) \leq s_{j}(n)/2 \label{eq:hats2}
\end{equation}
From equation \eqref{eq:lambda1} (see below), one has that
$
\mathbb{P}\left(\| \lambda \|_{1} \geq  \tilde{K}_{n} \right) \leq 2 n^{-\gamma},
$
which implies that $s_{j}(n) \leq \hat{s}_{j}(n)$ with probability larger than $1 - 2 n^{-\gamma}$. Hence, by inequalities  \eqref{eq:hats1} and \eqref{eq:hats2}, it follows that for all sufficiently large $n$
\begin{equation} \label{eq:hats3}
2 \max \left(\Delta_{jk}^n(\gamma), \sqrt{\frac{2 V_j^2 \gamma \log n}{n}} + \delta_j \frac{\gamma \log n}{3n} \right) \leq \hat{s}_{j}(n)/2
\end{equation}
with probability larger than $1 - 2 n^{-\gamma}$. Therefore, for all sufficiently large $n$, Proposition \ref{prop:dev} and inequality  \eqref{eq:hats3} imply that
\begin{equation}
\mathbb{P} \left(|\hat{\beta}_{j,k} - \beta_{j,k}|>   \hat{s}_{j}(n)/2 \right)  \leq  C n^{-\gamma}, \label{eq:randthr}
\end{equation}
for all $j_{0}(n) \leq j \leq j_{1}(n)$. \\

\noindent \underline{Bound on $R_{21} + R_{31}$:} thus, using the assumption that $\gamma \geq 2$, inequality  (\ref{eq:john}) and Proposition  \ref{prop:moments}, one has that for all sufficiently large $n$
$$
R_{21} + R_{31} \leq C  \frac{1}{n}  \left[ \sum_{j = j_{0}}^{j_{1}}  2^{j} \left( \frac{ 2^{4 j \nu}}{n^{2 }} \left( 1 + \frac{ 2^{j}}{n} \right) \right)^{1/2}   +  \sum_{j = j_{0}}^{j_{1}}   2^{-2js^{\ast}}  \right].
$$
By definition of $j_{1}$ one has that $\frac{ 2^{j}}{n} \leq C$ for all $j \leq j_{1}$, which implies that (since $s^{\ast} > 0$) 
\begin{equation}
R_{21} + R_{31} \leq C \frac{1}{n}  \left[ \sum_{j = j_{0}}^{j_{1}}  \frac{2^{j(2\nu +1)}}{n}  +  \sum_{j = j_{0}}^{j_{1}}   2^{-2js^{\ast}} \right]  = \opO( n^{-\frac{2s}{2s + 2 \nu +1}}) \label{eq:R21R31},
\end{equation}
using the fact that $\frac{2^{j(2\nu +1)}}{n} \leq C$ for all $j \leq j_{1}(n) \leq  \frac{1}{2 \nu +1} \log_{2} n$.

Finally, it remains to bound the term $T_{2} = R_{22} + R_{32}$. For this purpose, let $j_{2}$ be the largest integer such that
$
2^{j_{2}} \leq n^{\frac{1}{2s+2\nu+1}} (\log n)^{\beta} \mbox { with } \beta = -\frac{1}{2s + 2\nu +1},
$
and partition $T_{2}$ as $T_{2} = T_{21} + T_{22}$ where the first component $T_{21}$ is calculated over the resolution levels $j_{0} \leq j \leq j_{2}$ and the second  component $T_{22}$ is calculated over the resolution levels $j_{2} +1 \leq j \leq j_{1}$ (note that given our assumptions then $j_{2} \leq j_{1}$ for all sufficiently large $n$).  Using the definition of the threshold $\hat{s}_{j}(n)$ it follows that
\begin{equation}
\hat{s}_{j}(n)^{2}  \leq C \left( \sigma_{j}^2 (\| g \|_{\infty} \tilde{K}_{n} + \delta) \frac{\log(n)}{n} + \frac{(\log n)^2}{n^2} \epsilon_{j}^2 \right). \label{eq:boundthr}
\end{equation}
From Assumption \ref{ass:g} on the $\gamma_{\ell}$'s and equation \eqref{eq:Omegaj} for $\Omega_{j}$  it follows that
$$
\sigma_{j}^2 \leq C 2^{2j\nu}  \mbox{ and } \epsilon_{j} \leq C 2^{j(\nu+1/2)}.
$$
Since, for $2^j\frac{\log n}{n} \leq \left( \frac{\log n}{n} \right)^{-\frac{2 \nu}{ 2 \nu +1 }} $ all $j \leq j_{1}$, it follows that $\frac{(\log n)^2}{n^2} \epsilon_{j}^2 \leq C 2^{2j\nu} \frac{\log(n)}{n}$ and thus
\begin{equation}
\hat{s}_{j}(n)^{2}  \leq C 2^{2j\nu}  (\| g \|_{\infty} \tilde{K}_{n} + \delta + 1) \frac{\log(n)}{n}. \label{eq:boundthr}
\end{equation}
Using Proposition \ref{prop:moments}, the bound \eqref{eq:boundthr}, the fact that
\begin{equation}
\E \tilde{K}_{n} \leq \| \lambda_{1} \|_{1} + \opO  \left( \left( \frac{\log n}{n} \right)^{1/2} \right)  \label{eq:boundK}
\end{equation}
and the definition of $j_{2}$ one obtains that
\begin{eqnarray}
T_{21} \leq   \sum_{j = j_{0}}^{j_{2}}  \sum_{k=0}^{2^{j}-1} \left( \E ( \hat{\beta}_{j,k} - \beta_{j,k})^2   +     \frac{9}{4}  \E \hat{s}_{j}(n)^{2} \right)  & = &  \opO \left(  \frac{2^{j_{2}(2\nu+1)}}{n}  \log(n) \right) \nonumber  \\
& = & \opO \left( n^{-\frac{2s}{2s+2\nu+1}} (\log n)^{\frac{2s}{2s + 2\nu+1}} \right). \label{eq:boundT21}
\end{eqnarray}
Then, it remains to obtain a bound for $T_{22}$. Recall that
$
\hat{s}_{j}(n) \geq s_{j}(n)
$
with probability larger that $1 - 2 n^{-\gamma}$, where $s_{j}(n)$ is defined in \eqref{eq:thrfixed}. Therefore, by using Cauchy-Schwarz inequality one obtains that 
\begin{eqnarray*}
\E \left[ ( \hat{\beta}_{j,k} - \beta_{j,k})^2 \1_{\{ |\beta_{j,k}| \geq \hat{s}_{j}(n)/2  \}} \right] & \leq &   \E  ( \hat{\beta}_{j,k} - \beta_{j,k})^2 \1_{\{ |\beta_{j,k}| \geq s_{j}(n)/2   \}} \\
& & + \left( \E ( \hat{\beta}_{j,k} - \beta_{j,k})^{4} \right)^{1/2} \left(  \PP ( \hat{s}_{j}(n) \leq  s_{j}(n)  ) \right)^{1/2}
\end{eqnarray*}
Then, by Assumption \ref{ass:g} one has that $\sigma^2_{j} \geq C 2^{2j\nu}$. Therefore, using Proposition \ref{prop:moments}  it follows  that $ \E  ( \hat{\beta}_{j,k} - \beta_{j,k})^2 \leq C s^2_{j}(n)$ and that $ \E ( \hat{\beta}_{j,k} - \beta_{j,k})^{4} \leq C \frac{2^{4j\nu}}{{n^2}}$ for all $j \leq j_{1}$. Finally, using that $\gamma \geq 2$ and the fact that $\PP ( \hat{s}_{j}(n) \leq  s_{j}(n)  ) \leq 2 n^{-\gamma}$, one finally obtains that for any $j \leq j_{1}$
\begin{equation}
\E \left[ ( \hat{\beta}_{j,k} - \beta_{j,k})^2 \1_{\{ |\beta_{j,k}| \geq \hat{s}_{j}(n)/2  \}} \right] \leq C \left( \frac{s^2_{j}(n) }{4}  \1_{\{ |\beta_{j,k}| \geq s_{j}(n) /2  \}} + \frac{2^{2j \nu}}{n^{2}} \right) \label{eq:boundbeta}
\end{equation}
Let us first consider the case $p \geq 2$. Using inequality (\ref{eq:boundbeta})  one has that
\begin{eqnarray*}
T_{22} &  \leq &  C  \left( \sum_{j = j_{2}+1}^{j_{1}}  \sum_{k=0}^{2^{j}-1} \frac{ s^2_{j}(n) }{4} \1_{\{ |\beta_{j,k}| \geq  s_{j}(n)  /2 \}} + \frac{2^{2j \nu}}{n^{2}}   + |\beta_{j,k}|^{2}  \right) \\
& \leq & C \left(  \sum_{j = j_{2}+1}^{j_{1}}     \sum_{k=0}^{2^{j}-1}  |\beta_{j,k}|^{2}  + \frac{1}{n} \sum_{j = j_{2}+1}^{j_{1}}   \frac{2^{j (2\nu+1)}}{n}    \right).
\end{eqnarray*}
Then (\ref{eq:john}),  the definition of $j_{2}$, $j_{1}$ and the fact that $s^{\ast} = s$ imply that
\begin{equation}
T_{22}   = \opO \left( 2^{-2j_{2}s} + \frac{1}{n} \sum_{j = j_{2}+1}^{j_{1}}   \frac{2^{j (2\nu+1)}}{n}   \right) = \opO  \left( n^{-\frac{2s}{2s+2\nu+1}} (\log n)^{\frac{2s}{2s+2\nu+1}}  \right) \label{eq:boundT22_1}
\end{equation}
Now, consider the case $1 \leq p < 2$. Using again  inequality (\ref{eq:boundbeta})  one obtains that
\begin{eqnarray}
T_{22} &  \leq &  C \left( \sum_{j = j_{2}+1}^{j_{1}}  \sum_{k=0}^{2^{j}-1} \frac{s^{2}_{j}(n)}{4} \1_{\{ |\beta_{j,k}| \geq s_{j}(n)/2 \}}  + \frac{2^{2j \nu}}{n^{2}}   + \E |\beta_{j,k}|^{2}  \1_{\{ |\beta_{j,k}| < \frac{3}{2} \hat{s}_{j}(n) \}} \right) \nonumber \\
& \leq &   C \left(  \sum_{j = j_{2}+1}^{j_{1}}  \sum_{k=0}^{2^{j}-1} s_{j}(n)^{2-p}  |\beta_{j,k}|^{p} +  |\beta_{j,k}|^{p} \E \hat{s}_{j}(n)^{2-p}  + \frac{1}{n} \sum_{j = j_{2}+1}^{j_{1}}   \frac{2^{j (2\nu+1)}}{n}   \right) \label{eq:T22} 
\end{eqnarray}
By Hölder inequality, it follows that for any $\alpha > 1 $, $ \E \hat{s}_{j}(n)^{2-p} \leq  \left( \E \hat{s}_{j}(n)^{\alpha(2-p)} \right)^{1/\alpha} $. Hence, by taking $\alpha = 2/(2-p)$ it follows that $ \E \hat{s}_{j}(n)^{2-p} \leq \left( \E \hat{s}_{j}(n)^{2} \right)^{(2-p)/2} $. Then, using the following upper bounds (as a consequence of the definition of $ s^2_{j}(n)$ and the arguments used to derive inequalities \eqref{eq:boundthr}, \eqref{eq:boundK})
$$
s^2_{j}(n) \leq  C 2^{2j\nu} \frac{\log(n)}{n} \mbox{ and } \E \hat{s}_{j}(n)^{2}  \leq C 2^{2j\nu} \E \tilde{K}_{n} \frac{\log(n)}{n} \leq C 2^{2j\nu}  \frac{\log(n)}{n},
$$
it follows that inequality \eqref{eq:T22} and the fact that for $\lambda \in B_{p,q}^{s}(A)$,  $ \sum_{k=0}^{2^{j}-1}  |\beta_{j,k}|^{p} \leq C 2^{-jps^{\ast}}$ (with $ps^{\ast} =  ps + p/2-1$) imply that
\begin{eqnarray}
T_{22} 
&  \leq &   C \left( \sum_{j = j_{2}+1}^{j_{1}} 2^{2j\nu(1-p/2)} \left( \frac{\log(n)}{n} \right)^{1-p/2} 2^{-j ps^{\ast}}  + \frac{1}{n} \sum_{j = j_{2}+1}^{j_{1}}   \frac{2^{j (2\nu+1)}}{n}  \right) \nonumber \\
& \leq & C \left(  \left( \frac{\log n}{n} \right)^{1-p/2}  \sum_{j = j_{2}+1}^{j_{1}}  2^{j(\nu(2-p) - ps^{\ast})} + \frac{1}{n} \sum_{j = j_{2}+1}^{j_{1}}   \frac{2^{j (2\nu+1)}}{n}  \right)  \nonumber \\
&=& \opO \left(   \left( \frac{\log n }{n} \right)^{1-p/2}  2^{j_{2}(\nu(2-p) - ps^{\ast})} + \frac{1}{n} \sum_{j = j_{2}+1}^{j_{1}}   \frac{2^{j (2\nu+1)}}{n}   \right) \nonumber \\
& = &  \opO \left( n^{-\frac{2s}{2s+2\nu+1}}  (\log n)^{\frac{2s}{2s+2\nu+1}}\right) \label{eq:boundT22_2}
\end{eqnarray}
where we have used the assumption $\nu(2-p) < p s^{\ast}$ and the definition of $j_{2}$, $j_{1}$ for the last inequalities. Finally, combining the bounds     (\ref{eq:R4}), (\ref{eq:R1}), (\ref{eq:R21R31}), (\ref{eq:boundT21}), (\ref{eq:boundT22_1}) and (\ref{eq:boundT22_2}) completes the proof of Theorem \ref{theo:upperbound}.  \hfill $\square$

\subsubsection{Technical results}
\label{technicalupper}
Arguing as in the proof of Proposition 3 in \cite{MR2676894}, one has the following lemma:

\begin{lemma} \label{lem:psitilde}  Suppose that $g$ satisfies Assumption \ref{ass:g} and Assumption \ref{ass:gdecay}. Then, there exists a constants $C > 0$ such that for any $j \geq 0$ and $0 \leq k \leq 2^{j}-1$
$$
\| \tilde{\psi}_{j,k} \|_{\infty} \leq C 2^{j(\nu +1/2)} ,  \;  \| \tilde{\psi}_{j,k} \|^{2}_{2} \leq C 2^{2j\nu} \mbox{ and }   \| \tilde{\psi}_{j,k}^{2} \|^{2}_{2} \leq C 2^{j(4\nu+1)}. 
$$
\end{lemma}

\begin{proposition} \label{prop:moments}
There exists $C > 0$ such that for any $j \geq 0$ and $0 \leq k \leq 2^{j}-1$
\begin{equation}
\mathbb{E}|\hat{c}_{j,k} - c_{j,k}|^2 \leq C \frac{2^{2 j \nu}}{n} \left(1 +  \|\lambda \|_{2}  \| g\|_{\infty} \right), \quad
\mathbb{E}|\hat{\beta}_{j,k} - \beta_{j,k}|^2 \leq C \frac{2^{2 j \nu}}{n} \left(1 +  \|\lambda \|_{2}  \| g\|_{\infty} \right),
\end{equation}
and
\begin{equation}
\mathbb{E}|\hat{\beta}_{j,k} - \beta_{j,k}|^4 \leq C \frac{2^{4 j \nu}}{n^2} \left(1 +\frac{2^{j}}{n}\right) \left( 1 +  \|\lambda \|_{2}^2  \| g\|_{\infty}^2 + \|\lambda \|_{2}  \| g\|_{\infty} + \|\lambda \|_{2}^2  \| g\|_{\infty}   \right).
\end{equation}
\end{proposition}

\begin{proof}  We only prove the proposition for the wavelet coefficients $\hat{\beta}_{j,k}$ since the arguments are the same to prove the result for the scaling coefficients $\hat{c}_{j,k}$. Remark first that 
$
\hat{\beta}_{j,k} - \beta_{j,k} = \sum_{\ell  \in \Omega_j} c_\ell(\psi_{j,k}) (\hat{\theta}_\ell - \theta_\ell) =  Z_1 + Z_2,
$
where $Z_1$ and $Z_2$ are the centered variables
$$
Z_1: = \sum_{\ell  \in \Omega_j} (\tilde{\gamma}_\ell \gamma^{-1}_\ell -1) \theta_\ell c_\ell(\phi_{j,k}).
$$
and
$$
Z_2 :=  \frac{1}{n} 
\sum_{i=1}^n \int_{0}^1 \tilde{\psi}_{j,k} (t)  d \tilde{N}_t^{i} .
$$
where $d \tilde{N}_t^{i}  =  dN_t^{i} - \lambda(t-\btau_{i})dt$. \\

\noindent {\bf Control of the moments of $Z_{1}$}: by arguing as in the proof of Proposition 3 in \cite{MR2676894}, one obtains that  there exists a universal constant $C > 0$ such that
\begin{equation}\label{eq:Z1square}
\mathbb{E} |Z_1|^2 \leq C \frac{2^{2 j \nu}}{n}  \mbox{ and } \mathbb{E} |Z_1|^4 \leq C \left(\frac{2^{4 j \nu}}{n^2}+\frac{2^{j(4\nu+1)}}{n^3}\right).
\end{equation}
The main arguments  to obtain \eqref{eq:Z1square} rely on concentration inequalities on the variables $\btau_i,i=1,\ldots,n$.\\

\noindent
{\bf Control of the moments of $Z_{2}$}: using Lemma \ref{lem:psitilde} remark that
\begin{eqnarray*}
\mathbb{E} |Z_2|^2& =& \frac{1}{n^2} \sum_{i=1}^n \mathbb{E} \int_0^1 \tilde{\psi}^{2}_{j,k} (t)  \lambda(t-\btau_{i}) dt  = \frac{1}{n} \int_0^1\tilde{\psi}^{2}_{j,k} (t) \lambda \star g(t) dt,  \\
& \leq & C \frac{2^{2j \nu}}{n} \|\lambda \star g\|_{\infty} \leq C \frac{2^{2j \nu}}{n} \|\lambda \|_{2}  \| g\|_{\infty}.
\end{eqnarray*}
Let us now bound $\mathbb{E} |Z_2|^4$ by using Rosenthal's inequality \cite{MR0440354}
$$
\mathbb{E} \left| \sum_{i=1}^n Y_i\right|^{2p} \leq \left( \frac{16 p}{\log(2p)}\right)^{2p} \max \left\{\left(\sum_{i=1}^n \mathbb{E}  Y_i^2\right)^{p} ;  \sum_{i=1}^n \mathbb{E} |Y_i|^{2p}\right\},
$$
which is valid for independent, centered and real-valued random variables $(Y_i)_{i=1 \ldots,n}$.
We apply this inequality to
$
Y_i =   \int_{0}^{1} \tilde{\psi}_{j,k}(t) d\tilde{N}_t^{i}
$
with $p=2$. Conditionnaly to  $\btau_i$, using Proposition 6 in \cite{Reynaud} and the Jensen's inequality, it follows that
\begin{eqnarray*}
\E \left[ Y_i ^4 | \btau_{i} \right] & = &  \int_{0}^{1} \tilde{\psi}_{j,k}^{4}(t) \lambda(t-\btau_{i}) dt + 3 \left(   \int_{0}^{1} \tilde{\psi}_{j,k}^{2}(t) \lambda(t-\btau_{i}) dt \right)^2, \\
& \leq &  \int_{0}^{1} \tilde{\psi}_{j,k}^{4}(t) \left( \lambda(t-\btau_{i}) + 3 \lambda^{2}(t-\btau_{i}) \right) dt.
\end{eqnarray*}
Hence 
$
\mathbb{E} \sum_{i=1}^n Y_i ^4 \leq n  \int_{0}^{1} \tilde{\psi}_{j,k}^{4}(t) \left( \lambda \star g(t)  + 3\lambda^2 \star g(t) \right) dt.
$
Then, using Lemma \ref{lem:psitilde}
$
\mathbb{E} \sum_{i=1}^n Y_i ^4 \leq C n 2^{j(4\nu +1)} \left( \|\lambda \|_{2}  + \|\lambda \|_{2}^{2} \right)  \| g\|_{\infty}.
$
Using again Proposition 6 in \cite{Reynaud} and Lemma \ref{lem:psitilde} one obtains that $\mathbb{E}  Y_i^2 = \int_{0}^{1} \tilde{\psi}_{j,k}^{2}(t)  \lambda \star g(t) dt \leq C 2^{2j\nu}  \|\lambda \|_{2}    \| g\|_{\infty} $ which ends the proof of the proposition. \hfill $\square$
\end{proof} \\


\begin{proposition} \label{prop:dev}
Assume that $\lambda \in \Lambda_{\infty}$ and let $\gamma > 0$.  Then, there exists a constant $C >0$ such that for any $j \geq 0$, $k \in \{0 \dots 2^{j}-1\}$ and all sufficiently large $n$
\begin{equation}
\mathbb{P} \left(|\hat{\beta}_{j,k} - \beta_{j,k}|> 2 \max \left(\Delta^n_{jk}(\gamma)) , \sqrt{\frac{2 V_j^2 \gamma \log n}{n}} + \delta_j \frac{\gamma \log n}{3n} \right)\right)  \leq C n^{-\gamma},
\end{equation}
where $\Delta^n_{jk}(\gamma)$ is defined in (\ref{eq:deltajkn}).
\end{proposition}

\begin{proof}

Using the notations introduced in the proof of Proposition \ref{prop:moments}, write $\hat{\beta}_{j,k}-\beta_{j,k} = Z_1+Z_2$ and remark that for any $u > 0$
\begin{equation}
\mathbb{P}(|Z_1+Z_2|> u) \leq \mathbb{P}(|Z_1|> u/2)+\mathbb{P}(|Z_2|> u/2) \label{eq:Z1Z2}
\end{equation}
Now, arguing as in Proposition 4 in \cite{MR2676894} and using Bernstein's inequality,  one has immediately that 
\begin{equation}
\mathbb{P}\left( |Z_1|> \sqrt{\frac{2 V_j^2 \gamma\log n}{n}} + \delta_j \frac{\gamma\log n}{3n}\right) \leq 2n^{-\gamma}. \label{eq:Z1}
\end{equation}

Let us now control the deviation of $Z_2  = \frac{1}{n} \sum_{i=1}^n \int_{0}^1 \tilde{\psi}_{j,k} (t) d\tilde{N}_t^{i}$. First, remark that conditionnaly to the shifts $\btau_1,\dots,\btau_n$, the process $\sum_{i=1}^n N^i$ is a Poisson process with intensity $\sum_{i=1}^n \lambda(.-\btau_i)$. For the sake of convenience, we introduce some additionnal notations. For $n \geq 1$, $j \geq 0$ and $0 \leq k \leq 2^{j}-1$, 
define
$$ M_{jk}^n = \frac{1}{n} \sum_{i=1}^n \int_0^1 \tilde \psi_{jk}^2(t) \lambda(t-\btau_i)dt, \ \mathrm{and} \ M_{jk} = \E M_{jk}^n = \int_0^1 \tilde \psi^2_{jk}(t) \lambda \star g(t)dt .
$$  

Using  an analogue of Bennett's inequality for Poisson processes (see e.g. Proposition 7 in \cite{Reynaud}), we get that for any $s>0$
\begin{equation} \label{eq:Bennett}
\PP \left( |Z_2|    >  \sqrt{\frac{2 s}{n} M_{jk}^n}     +  \frac{s}{3 n}     \| \tilde{\psi}_{j,k}\|_{\infty} \big| \btau_{1},\ldots,\btau_{n} \right) \leq 2 \exp\left(-s \right)
\end{equation}
Remark that the quantity $M_{jk}^n$ is not computable from the data as its depends on $\lambda$ and the unobserved shifts $\btau_{1},\ldots,\btau_{n}$. Nevertheless it is possible to compute a data-based upper bound for $M_{jk}^n$. Indeed,  note that Bernstein's inequality (see e.g. Proposition 2.9 in \cite{massart}) implies that 
$$
\mathbb{P} \left( M_{jk}^n > M_{jk} + \tilde{M}_{jk} \left( \frac{\gamma \log n}{ 3 n} + \sqrt{ \frac{2 \gamma \log n}{ n}  } \right)  \right) \leq n^{-\gamma}.
$$
with $\tilde{M}_{jk} = \|\lambda\|_{\infty} \|\tilde{\psi}_{j,k}\|_2^2$. Obviously, $\tilde{M}_{jk}$ is  unknown but for all sufficiently large $n$, one has that
$$
\tilde{M}_{jk} = \|\lambda\|_{\infty} \|\tilde{\psi}_{j,k}\|_2^2 \leq \log n \|\tilde{\psi}_{j,k}\|_2^2.
$$
Moreover, remark that
$
M_{jk} =  \|\psi_{jk} \sqrt{\lambda\star g} \|_2^2  \leq \|\psi_{jk} \|_2^2 \| g\|_{\infty} \|\lambda\|_1.
$
Hence, 
\begin{equation} \label{eq:Mjkn}
\mathbb{P} \left( M_{jk}^n > \| \tilde{\psi}_{j,k} \|_{2}^{2} \left( \| g \|_{\infty} \| \lambda \|_{1}  + \left( \frac{\gamma (\log n)^2}{ 3 n} + \sqrt{ \frac{2 \gamma (\log n)^3 }{ n}  } \right) \right) \right) \leq n^{-\gamma}.
\end{equation}
To obtain a data-based upper bound for $M_{jk}^n$, it remains to obtain an upper bound for $\| \lambda \|_{1}$. Recall that we have denoted by $K_{i}$ the number of points of the process $N^i$. Conditionally to $\btau_{i}$,  $K_{i}$ is real random variable that follows a Poisson distribution with intensity $\int_{0}^{1} \lambda(t-\btau_{i}) dt$. Since $\lambda$ is assumed to be periodic with period 1, it follows that for any $i=1,\ldots,n$, $\int_{0}^{1} \lambda(t-\btau_{i}) dt =  \int_{0}^{1} \lambda(t) dt$, and thus $(K_{i})_{i=1,\ldots,n}$ are i.i.d.\ random variables following  a Poisson distribution with intensity  $\| \lambda \|_{1} = \int_{0}^{1} \lambda(t) dt$. Using standard arguments to derive concentration inequalities one has that for any $u > 0$
$$
\mathbb{P}\left(\| \lambda \|_{1} \geq  \frac{1}{n}  \sum_{i=1}^{n} K_{i} + \sqrt{\frac{2 u \| \lambda \|_{1} }{n}} + \frac{u}{3 n}  \right) \leq 2 \exp(-u),
$$ 
where $\| \lambda \|_{1} =  \int_{0}^{1} \lambda(t) dt$. Now, define the function $h(y) = y^2-\sqrt{2a}y-a/3$ for $y \geq 0$ and with $a = u/n$. Then, the above inequality can be written as
$$
\mathbb{P}\left(h\left( \sqrt{\| \lambda \|_{1}} \right) \geq  \frac{1}{n}  \sum_{i=1}^{n} K_{i}  \right) \leq 2 \exp(-u).
$$
Since $h$ restricted on $[\sqrt{a}(\sqrt{30}+3\sqrt{2})/6;+\infty[$ is invertible with $h^{-1}(y) = \sqrt{y + \frac{5 a}{6}} + \sqrt{\frac{a}{2}}$ it follows that for $u =  \gamma \log n$ and all sufficiently large $n$
\begin{equation} \label{eq:lambda1}
\mathbb{P}\left(\| \lambda \|_{1} \geq  \bar{K}_{n} + \frac{4 \gamma \log n }{3n} + \sqrt{\frac{2  \gamma \log n }{n}  \bar{K}_{n}  + \frac{5  \gamma^{2} (\log n )^{2}}{3n^2}  } \right) \leq 2 n^{-\gamma},
\end{equation}
where $\bar{K}_{n} = \frac{1}{n}  \sum_{i=1}^{n} K_{i}$. Therefore, using \eqref{eq:Mjkn} it follows that
\begin{equation} \label{eq:Mjkn2}
\mathbb{P} \left( M_{jk}^n > \| \tilde{\psi}_{j,k} \|_{2}^{2} \left( \| g \|_{\infty}  \tilde{K}_n(\gamma) + \left( \frac{\gamma (\log n)^2}{ 3 n} + \sqrt{ \frac{2 \gamma (\log n)^3 }{ n}  } \right) \right) \right) \leq 3 n^{-\gamma},
\end{equation}
where $ \tilde{K}_n(\gamma)$ is defined in \eqref{eq:tildeK}. Hence, combining \eqref{eq:Bennett} with $s = \gamma\log n$ and \eqref{eq:Mjkn2} we obtain that
\begin{equation} \label{eq:Z2}
\PP \left( |Z_2|    >  \sqrt{\frac{2 \gamma\log n}{n}  \| \tilde{\psi}_{j,k} \|_{2}^{2} \left( \| g \|_{\infty}  \tilde{K}_n(\gamma) + \left( \frac{\gamma (\log n)^2}{ 3 n} + \sqrt{ \frac{2 \gamma (\log n)^3 }{ n}  } \right) \right)  }     +  \frac{\gamma\log n}{3 n}     \| \tilde{\psi}_{j,k}\|_{\infty} \right) \leq 5 n^{-\gamma}
\end{equation}
Combining inequalities \eqref{eq:Z1Z2},  \eqref{eq:Z1} and \eqref{eq:Z2} concludes the proof. \hfill $\square$ \\

\end{proof}


\section{Lower bound on the minimax risk}\label{sec:lowerbound}

\subsection{Main result}

\begin{theo}
\label{th:borneinf}
Suppose that  $g$ satisfies Assumption \ref{ass:g} and Assumption \ref{ass:gdecay}. Introduce the class of functions
 $$
 \Lambda_{0} = \left\{ \lambda \in L^{2}([0,1]); \; \lambda(t) \geq 0  \mbox{ for all } t \in [0,1] \right\}.
 $$
Let $1 \leq p \leq \infty$, $1 \leq q \leq \infty$, $A > 0$ and assume that
 $
 s > 2 \nu +1.
 $
Then, there exists a constant  $C_0>0$ (independent of $n$) such that for all sufficiently large $n$
$$ \inf_{\hat \lambda_{n}} \sup_{\lambda \in B^s_{p,q}(A) \bigcap  \Lambda_{0}} \mathcal{R}(\hat{\lambda}_{n},\lambda) \geq C_0 n^{-\frac{2s}{2s+2\nu+1}} ,$$
where the above infimum is taken over the set of all possible estimators $\hat\lambda_{n} \in L^{2}([0,1]) $ of the  intensity $\lambda$ (i.e the set of all measurable mapping of the random processes $N^{i}, \; i=1,\ldots,n$ taking their value in $L^{2}([0,1])$).
\end{theo}

\subsection{Some properties of Meyer wavelets}

Recall that the Meyer mother wavelet $\psi$ is not compactly supported. Nevertheless, Meyer wavelet function satisfies the following proposition which will be useful for the construction of a lower bound of the minimax risk.

\begin{proposition}\label{bornemeyer}
There exists a universal constant $c(\psi)$ such that for any $j \in \mathbb{N}$ and for any $(\omega_k)_{0 \leq k \leq 2^j -1} \in \{0,  1\}^{2^{j}}$
$$
\sup_{x \in \mathbb{R}}  \left| \sum_{k=0}^{2^j-1} \omega_k \psi_{j,k}(x) \right| \leq c(\psi) 2^{j/2}.
$$
\end{proposition}

\begin{proof}
Note that for periodic Meyer wavelets, one has that
$$
\sup_{x \in \mathbb{R}} \sum_{k \in \mathbb{Z}} |\psi(x-k)| < \infty.
$$ 
Hence the proof follows using the definition of $\psi_{j,k}(x) = 2^{j/2} \psi(2^j x - k )$.\cqfd
\end{proof}

\subsection{Definitions and notations}
Recall that $\btau_{1},\ldots,\btau_{n}$ are i.i.d.\ random variables with density $g$, and that for $\lambda \in  \Lambda_{0}$ a given intensity, we denote by $N^1,\ldots,N^n$  the  counting processes such that conditionally to $\btau_{1},\ldots,\btau_{n}$, $N^1,\ldots,N^n$ are independent Poisson processes with intensities $\lambda(\cdot -\btau_{1}),\ldots,\lambda(\cdot -\btau_{n})$. 
Then, the notation $\mathbb{E}_{\lambda}$ will be used to denote the expectation with respect to the distribution $\PP_{\lambda}$ (tensorized law) of the multivariate counting process $N = \left( N^1, \dots , N^n \right)$. In the rest of the proof, we also assume that $p,q$ denote two integers such that $1 \leq p \leq \infty$, $1 \leq q \leq \infty$, $A$ is a positive constant, and  that $s$ is a positive real such that $ s > 2 \nu+1 $, where $\nu$ is the degree of ill-posedness defined in Assumption \ref{ass:g}.

A key step in the proof is the use of the likelihood ratio $\Lambda(H_0,H_1)$  between two measures associated to two  hypotheses $H_0$ and $H_1$ on the  intensities of the Poisson processes we consider. The following lemma, whose proof can be found in \cite{Bremaud}, is a Girsanov's like formula for Poisson processes.

\begin{lemma} [Girsanov's like formula]
\label{bremaud}
Let $\mathcal{N}_0$ (hypothesis $H_0$) and $\mathcal{N}_1$ (hypothesis $H_1$) two Poisson processes having respective intensity $\lambda_{0}(t) = \rho$ and $\lambda_{1}(t) = \rho+\mu(t) $  for all $t \in [0,1]$, where $\rho>0$ is a positive constant and $\mu \in  \Lambda_{0} $ is a positive function. Let $\PP_{\lambda_{1}}$ (resp. $\PP_{\lambda_{0}}$) be the distribution of $\mathcal{N}_1$ (resp. $\mathcal{N}_0$).  Then, the likelihood ratio between $H_0$ and $H_1$ is
\begin{equation} \label{eq:Girsa1}
\Lambda(H_0,H_1)(\mathcal{N}) := \frac{d \PP_{\lambda_{1}}}{d \PP_{\lambda_{0}}} (\mathcal{N}) =  \exp \left[ - \int_0^1 \mu(t) dt + \int_0^1 \log \left( 1+ \frac{\mu(t)}{\rho} \right) d\mathcal{N}_{t} \right],
\end{equation}
where  $\mathcal{N}$ is a Poisson process with intensity belonging to $ \Lambda_{0}$.
\end{lemma}
The above lemma means that if  $F(\mathcal{N})$ is a real-valued and bounded measurable function of the counting process  $\mathcal{N} = \mathcal{N}_1$ (hypothesis $H_1$), then
\begin{eqnarray*}
\mathbb{E}_{H_{1}} \left[ F(\mathcal{N}) \right]  = \mathbb{E}_{H_{0}} \left[F(\mathcal{N}) \Lambda(H_0,H_1)(\mathcal{N})  \right] 
\end{eqnarray*}
where $\mathbb{E}_{H_{1}}$ denotes the expectation with respect to $\PP_{\lambda_{1}}$  (hypothesis $H_1$), and $\mathbb{E}_{H_{0}}$ denotes the expectation with respect to $\PP_{\lambda_{0}}$ (hypothesis $H_0$).

Obviously, one can adapt Lemma \ref{bremaud} to the case of $n$ independent  Poisson processes $\mathcal{N}=(\mathcal{N}^1, \dots \mathcal{N}^n)$ with respective intensities $\lambda_{i}(t) = \rho+\mu_{i}(t), \; t \in [0,1], i=1,\ldots,n$ under $H_{1}$ and  $\lambda_{i}(t) = \rho , \; t \in [0,1], i=1,\ldots,n$ under $H_{0}$, where $\mu_{1},\ldots,\mu_{n}$ are  positive intensities in $  \Lambda_{0}$. In such a case, the Girsanov's like formula \eqref{eq:Girsa1} becomes
\begin{equation} \label{eq:Girsan}
\Lambda(H_0,H_1)(\mathcal{N}) = \prod_{i=1}^n \exp \left[ - \int_0^1 \mu_{i}(t) dt + \int_0^1 \log \left( 1+ \frac{\mu_{i}(t)}{\rho} \right) d\mathcal{N}^i_t \right].
\end{equation}

\subsection{Minoration of the minimax risk using the Assouad's cube technique}

Let us first describe the main idea of the proof. In Lemma \ref{lowertrick}, we  provide a first result giving a lower bound on the quadradic risk of any estimator over a specific set of test functions. These test functions are appropriate linear combinations of  Meyer wavelets whose construction follows ideas of the Assouad's cube technique to derive lower bounds for minimax risks (see e.g.\ \cite{HKPT,MR2488345}). A key step in the proof of  Lemma \ref{lowertrick} is the use of the likelihood ratio formula \eqref{eq:Girsan}. Then, we detail precisely in Lemma \ref{calcul} the asymptotic behavior of the likelihood ratio \eqref{eq:r} defined in   Lemma \ref{lowertrick}   under well-chosen hypotheses  $H_{1}$ and $H_{0}$. The result of Theorem \ref{th:borneinf} then follows from these two lemmas.

Given an integer $D \geq 1$, introduce $$ S_D(A) = \{f \in  \Lambda_{0} \cap  B^{s}_{p,q}(A)  \quad | \quad \langle f,\psi_{j,k} \rangle = 0 \ \forall j\not = D\ \forall k \in \{0 \dots 2^{j}-1\}   \}.$$
For any $\omega = (\omega_{k})_{k=0,\ldots,2^D-1} \in \lbrace 0,1 \rbrace^{2^D}$ and $\ell\in \lbrace 0,\dots,2^D-1\rbrace$, we define $\bar {\omega}^{\ell}\in \lbrace 0,1 \rbrace^{2^D}$ as $\bar{\omega}^{\ell}_k = \omega_k, \forall k \neq l$
 and $\bar{\omega}^{\ell}_\ell  = 1-\omega_{\ell}$. In what follows, we will use the likelihood ratio formula \eqref{eq:Girsan} with the intensity
 \begin{equation}
\lambda_{0}(t) = \rho(A) = \frac{A}{2}, \forall t \in [0,1], \label{eq:lambda0}
\end{equation}
 which corresponds to the hypothesis $H_{0}$ under which  all the intensities of  the observed counting processes are constant and equal to $A/2$ where $A$ is the radius of the Besov ball $B^{s}_{p,q}(A)$.
Next, for any  $\omega\in \lbrace 0,1 \rbrace^{2^D-1}$,
we denote by $\lambda_{D,\omega}$ the intensity defined as
\begin{equation}
\lambda_{D,\omega} = \rho(A) + \xi_D \sum_{k=0}^{2^D-1} w_k \psi_{D,k} + \xi_D 2^{D/2} c(\psi), \ \mathrm{with} \ \xi_D= c 2^{-D(s+1/2)},
\label{eq:fonctionstests}
\end{equation}
for some constant $0 < c \leq  A / (2+ c(\psi) )$, and  where $c(\psi)$ is the constant introduced in Proposition \ref{bornemeyer}. 
For the sake of convenience, we omit in what follows the subscript $D$ and write $\lambda_{\omega}$ instead of $\lambda_{D,\omega}$. First, remark that each function $\lambda_{\omega}$ can be written as $\lambda_{\omega} = \rho(A) + \mu_{\omega}$ where
$$
\mu_{\omega} = \xi_D \sum_{k=0}^{2^D-1} w_k \psi_{D,k} + \xi_D 2^{D/2} c(\psi),
$$
is a positive intensity belonging to $\Lambda_{0}$ by Proposition \ref{bornemeyer}. Moreover, it can be checked that the condition $c  \leq A / (2+ c(\psi) )$ implies that $\lambda_{\omega} \in B^{s}_{p,q}(A)$. Therefore, $\lambda_{\omega} \in S_D(A)$ for any $\omega \in \lbrace 0,1 \rbrace^{2^D}$. The following lemma provides a lower bound on $S_D$.

\begin{lemma}
\label{lowertrick}
Using the notations defined above,  the following inequality holds
$$\inf_{\hat \lambda_{n}} \sup_{\lambda \in S_D(A)} \mathbb{E}_{\lambda} \| \hat \lambda_{n} - \lambda \|^2 \geq 
\frac{\xi_D^2}{4}\frac{1}{2^{2^D}}   \sum_{k=0}^{2^D-1}  \sum_{\omega\in \lbrace 0,1 \rbrace^{2^D} | w_{k}=1}  \mathbb{E}_{\lambda_{\omega}} \left[ 1 \wedge \mathcal{Q}_{k,\omega}(N) \right], $$
with $N = \left( N^1, \ldots , N^n \right)$ and

  \begin{equation}\label{eq:r}
\mathcal{Q}_{k,\omega}(N) = \frac{ \int_{\R^{n}} \prod_{i=1}^n \exp\left[ -\int_0^1 \mu_{\bar{\omega}^k}(t-\alpha_i)dt + \int_0^1 \log \left( 1+ \frac{\mu_{\bar{\omega}^k}(t-\alpha_i)}{\rho(A)} \right)  dN^{i}_{t} \right] g(\alpha_{i}) d \alpha_{i} }{ \int_{\R^{n}}  \prod_{i=1}^n \exp\left[ -\int_0^1 \mu_{\omega}(t-\alpha_i)dt + \int_0^1 \log \left( 1+ \frac{\mu_{\omega}(t-\alpha_i)}{\rho(A)} \right) dN^{i}_{t} \right] g(\alpha_{i}) d \alpha_{i} }.\end{equation}
\end{lemma}

\begin{proof}
Let $\hat \lambda_{n} = \hat \lambda_{n}(N ) \in L^{2}([0,1])$ denote any estimator of $\lambda \in S_D(A)$ (a measurable function of the process $N$). Note that, to simplify the notations, we will drop in the proof the dependency of $\hat \lambda_{n}(N )$ on $N$ and $n$, and we write $\hat \lambda = \hat \lambda_{n}(N )$. 
Then,  define
$$
R(\hat \lambda ) =  \sup_{\lambda \in S_D(A)} \mathbb{E}_{\lambda} \| \hat \lambda - \lambda \|^2.
$$
Since $\lambda_{\omega} \in S_D(A)$ for any $\omega \in \lbrace 0,1 \rbrace^{2^D}$, it follows from Parseval's relation that
\begin{equation*}
R(\hat \lambda ) \geq   \sup_{\omega \in \lbrace 0,1 \rbrace^{2^D} } \mathbb{E}_{\lambda_{\omega}} \| \hat \lambda - \lambda_{\omega} \|^2
\geq  \sup_{\omega \in \lbrace 0,1 \rbrace^{2^D} } \mathbb{E}_{\lambda_{\omega}} \sum_{k=0}^{2^D-1} | \beta_{D,k}(\hat \lambda) - \omega_k \xi_D |^2,
\end{equation*}
where we have used the notation $\beta_{D,k}(\hat\lambda) = \langle \hat\lambda, \psi_{D,k} \rangle$. For all $k\in \lbrace 0,\dots,2^D-1 \rbrace$ define
$$ \hat \omega_k = \hat \omega_k(N) := \arg \min_{v \in \lbrace 0,1 \rbrace} |\beta_{D,k}(\hat\lambda(N)) -v \xi_D |.$$
Then, the triangular inequality and the definition of $\hat \omega_k$ imply that
$$ \xi_D | \hat\omega_k  - \omega_k  | \leq  |\hat \omega_k \xi_D - \beta_{D,k}(\hat\lambda) | + | \beta_{D,k}(\hat\lambda) - \omega_k \xi_D | \leq 2 | \beta_{D,k}(\hat\lambda) - \omega_k \xi_D  |.   $$
Thus, 
\begin{eqnarray}
R(\hat \lambda )  & \geq & \frac{\xi_D^2}{4}  \sup_{\omega\in \lbrace 0,1 \rbrace^{2^D}} \mathbb{E}_{\lambda_{\omega}} \sum_{k=0}^{2^D-1} |\hat\omega_k(N ) -\omega_k|^2, \nonumber \\
& \geq & \frac{\xi_D^2}{4}\frac{1}{2^{2^D}} \sum_{\omega\in \lbrace 0,1 \rbrace^{2^D}} \sum_{k=0}^{2^D-1} \mathbb{E}_{\lambda_{\omega}} | \hat\omega_k(N  ) - \omega_k |^2. \label{eq:lowerboundR1}
\end{eqnarray}
Let $k\in \lbrace 0,\dots,2^D-1 \rbrace$ and $\omega\in \lbrace 0,1 \rbrace^{2^D}$ be fixed parameters. Conditionally to the vector $\btau=(\btau_1, \dots \btau_n) \in \R^{n}$, we define the two hypothesis $H_0$ and $H_{\omega}^{\btau}$ as
\begin{itemize}
\item[$H_0$:] $N^1,\ldots,N^n$ are independent Poisson processes with intensities $\left( \lambda_{0}(\cdot -\btau_{1}),\ldots,\lambda_{0}(\cdot -\btau_{n}) \right) = \left( \lambda_{0}(\cdot),\ldots,\lambda_{0}(\cdot) \right)$, where $\lambda_{0}$ is the constant intensity defined by \eqref{eq:lambda0},
\item[$H_{\omega}^{\btau}$:] $N^1,\ldots,N^n$ are independent Poisson processes with intensities $(\lambda_{\omega}(\cdot -\btau_1), \dots , \lambda_{\omega}(\cdot -\btau_n))$.
\end{itemize}
In what follows, we use the notation  $\E_{H_0}$ (resp.\  $\E_{H_{\omega}^{\btau}}$) to denote the expectation under the hypothesis $H_0$ (resp.\ $H_{\omega}^{\btau}$) conditionally to $\btau = (\btau_1, \dots \btau_n)$.  The Girsanov's like formula \eqref{eq:Girsan} yields
\begin{eqnarray*}
\mathbb{E}_{\lambda_{\omega}} | \hat\omega_k(N) - \omega_k |^2 &=& \int_{\R^{n}} \E_{H_1^{\tau}}   | \hat\omega_k(N) - \omega_k |^2 g(\tau_{1}) \ldots g(\tau_{n}) d \tau \\
& =&  \int_{\R^{n}}  \E_{H_0}  \left[ |\hat\omega_k(N) - \omega_k |^2 \Lambda(H_0,H_{\omega}^{\tau})(N)  \right] g(\tau_{1}) \ldots g(\tau_{n}) d \tau,
\end{eqnarray*}
with $d \tau = d \tau_1, \ldots, d \tau_n $ and
$$
\Lambda(H_0,H_{\omega}^{\tau})(N) = \prod_{i=1}^n \exp \left[ - \int_0^1 \mu_{\omega}(t- \tau_{i}) dt + \int_0^1 \log \left( 1+ \frac{\mu_{\omega}(t- \tau_{i})}{\rho(A)} \right) dN^i_{t} \right],
$$
for $N = (N^{1},\ldots,N^{n})$. Now, remark that under the hypothesis $H_0$, the law of the random variable $\hat{\omega}_k(N)$ does not depend on the random shifts $\btau = (\btau_{1},\ldots,\btau_{n})$ since $\lambda_{0}$ is a constant intensity. Thus, we obtain the following equality
\begin{equation} \label{eq:key}
\mathbb{E}_{\lambda_{\omega}} | \hat\omega_k(N) - \omega_k |^2 
= \E_{H_0}   \left[ |\hat\omega_k(N) - \omega_k |^2  \int_{\R^{n}} \Lambda(H_0,H_{\omega}^{\tau})(N) g(\tau_{1}) \ldots g(\tau_{n}) d \tau \right].
\end{equation}
 Using equality \eqref{eq:key}, we may re-write the lower bound \eqref{eq:lowerboundR1} on $R(\hat \lambda )$ as
 \begin{eqnarray*}
 R(\hat \lambda ) & \geq & \frac{\xi_D^2}{4}\frac{1}{2^{2^D}} \sum_{\omega\in \lbrace 0,1 \rbrace^{2^D}} \sum_{k=0}^{2^D-1}  \E_{H_0}   \left[ |\hat\omega_k(N) - \omega_k |^2  \int_{\R^{n}} \Lambda(H_0,H_{\omega}^{\tau})(N) g(\tau_{1}) \ldots g(\tau_{n}) d \tau \right] \\
 & = & \frac{\xi_D^2}{4}\frac{1}{2^{2^D}}  \sum_{k=0}^{2^D-1}  \sum_{\omega\in \lbrace 0,1 \rbrace^{2^D} | w_{k}=1}  \left( \E_{H_0}   \left[ |\hat\omega_k(N) - \omega_k |^2  \int_{\R^{n}} \Lambda(H_0,H_{\omega}^{\tau})(N) g(\tau_{1}) \ldots g(\tau_{n}) d \tau \right]  \right.+ \\
 & & \left.\E_{H_0}   \left[ |\hat\omega_k(N) - \bar{\omega}_k^{k} |^2  \int_{\R^{n}} \Lambda(H_0,H_{\bar{\omega}^{k}}^{\tau})(N) g(\tau_{1}) \ldots g(\tau_{n}) d \tau \right] \right).
  \end{eqnarray*}
Using the inequality $|1-v|^2 z + |v|^2 z' \geq  z \wedge z' $ that holds for all $v \in \{0,1\}$ and all reals $z,z'>0$, we deduce that 
  \begin{eqnarray*}
 R(\hat \lambda ) & \geq &  \frac{\xi_D^2}{4}\frac{1}{2^{2^D}}   \sum_{k=0}^{2^D-1}  \sum_{\omega\in \lbrace 0,1 \rbrace^{2^D} | w_{k}=1} \E_{H_0} \left\{
\int_{\R^{n}} \Lambda(H_0,H_{\omega}^{\tau})(N) g(\tau_{1}) \ldots g(\tau_{n}) d \tau  \wedge \right., \\
& & \left. 
\int_{\R^{n}} \Lambda(H_0,H_{\bar{\omega}^{k}}^{\tau})(N) g(\tau_{1}) \ldots g(\tau_{n}) d \tau \right\} \\ 
& \geq  & \frac{\xi_D^2}{4}\frac{1}{2^{2^D}}    \sum_{k=0}^{2^D-1}  \sum_{\omega\in \lbrace 0,1 \rbrace^{2^D} | w_{k}=1} \E_{H_0} \int_{\R^{n}} \Lambda(H_0,H_{\omega}^{\tau})(N) g(\tau_{1}) \ldots g(\tau_{n}) d \tau  \left( 1 \wedge \right. \\ 
& & \left. \frac{\int_{\R^{n}} \Lambda(H_0,H_{\bar{\omega}^{k}}^{\alpha})(N) g(\alpha_{1}) \ldots g(\alpha_{n}) d \alpha}{\int_{\R^{n}} \Lambda(H_0,H_{\omega}^{\alpha})(N) g(\alpha_{1}) \ldots g(\alpha_{n}) d \alpha} \right), \\
& \geq  & \frac{\xi_D^2}{4}\frac{1}{2^{2^D}}     \sum_{k=0}^{2^D-1}  \sum_{\omega\in \lbrace 0,1 \rbrace^{2^D} | w_{k}=1}  \int_{\R^{n}}  \E_{H_0} \left[ \Lambda(H_0,H_{\omega}^{\tau})(N)  \left( 1 \wedge \mathcal{Q}_{k,\omega}(N) \right) \right]  g(\tau_{1}) \ldots g(\tau_{n}) d \tau, 
 \end{eqnarray*}
 where
 $$
 \mathcal{Q}_{k,\omega}(N) = \frac{\int_{\R^{n}} \Lambda(H_0,H_{\bar{\omega}^{k}}^{\alpha})(N) g(\alpha_{1}) \ldots g(\alpha_{n}) d \alpha}{\int_{\R^{n}} \Lambda(H_0,H_{\omega}^{\alpha})(N) g(\alpha_{1}) \ldots g(\alpha_{n}) d \alpha},
 $$
 and $d \alpha = d \alpha_{1} \ldots d \alpha_{n} $. Then, using again the  Girsanov's like formula   \eqref{eq:Girsan}, we obtain  the lower bound
$$
 R(\hat \lambda ) \geq  \frac{\xi_D^2}{4}\frac{1}{2^{2^D}}    \sum_{k=0}^{2^D-1}  \sum_{\omega\in \lbrace 0,1 \rbrace^{2^D} | w_{k}=1} \mathbb{E}_{ \lambda_{\omega} } \left[ 1 \wedge  \mathcal{Q}_{k,\omega}\right],
$$
that is independent of $\hat \lambda$ which ends the proof of the lemma. \hfill $\square$ \\

\end{proof}

\noindent We detail in the next paragraph how to use Lemma \ref{lowertrick}  with a suitable value for the parameter $D$ to obtain the desired lower bound on the minimax risk.

\subsection{Quantitative settings}

In the rest of the proof, we will suppose that $D = D_{n}$ satisfies the asymptotic equivalence
\begin{equation} \label{eq:Dn}
2^{D_{n}} \sim n^{\frac{1}{2s+2\nu+1}} \mbox{ as } n \to + \infty.
\end{equation}
To simplify the notations we will  drop the subscript $n$, and we write $D = D_{n}$. For two sequences of reals $(a_{n})_{n \geq 1}$ and $(b_{n})_{n \geq 1}$ we use the notation $a_{n} \asymp b_{n}$ if there exists two positive constants $C,C' > 0$ such that $C \leq \frac{a_{n}}{b_{n}} \leq C'$ for all sufficiently large $n$. 
Then, define $m_{D_{n}} = 2^{D_{n}/2} \xi_{D_{n}}$. Since $\xi_{D_{n}}= c 2^{-D_{n}(s+1/2)}$, it follows that  $$m_{D_{n}}   \asymp n^{-s/(2s+2\nu+1)} \to 0$$ as $n \to \infty$. Remark also that the condition $s>2 \nu+1$ implies that $$n m_{D_{n}}^3    \asymp n^{-(s-2\nu -1)/(2s+2\nu+1)} \to 0$$ as $n \to \infty$.
\subsection{Lower bound of the ``likelihood ratio" $ \mathcal{Q}_{k,\omega}$}

The above quantitative settings combined with Lemma \ref{lowertrick} will allow us to obtain a lower bound of the minimax risk. For this purpose, let $0 < \delta < 1$, and remark that Lemma  \ref{lowertrick} implies that
\begin{equation} \label{eq:boundS}
\inf_{\hat \lambda_{n}} \sup_{\lambda \in S_D(A)} \mathbb{E}_{\lambda} \| \hat \lambda - \lambda \|^2  \geq    \frac{\delta \xi_D^2}{4}\frac{1}{2^{2^D}}    \sum_{k=0}^{2^D-1}  \sum_{\omega\in \lbrace 0,1 \rbrace^{2^D} | w_{k}=1}  \PP_{\lambda_{\omega}} \left( \mathcal{Q}_{k,\omega}(N) \geq \delta \right).
\end{equation}
The remainder of the proof is thus devoted to the construction of a lower bound in probability for the random variable $\mathcal{Q}_{k,\omega}(N) := \frac{I_1}{I_2}$ where
$$
I_1 = I_1(N) =  \int_{\R^{n}} \prod_{i=1}^n \exp\left[ -\int_0^1 \mu_{\bar{\omega}^k}(t-\alpha_i)dt + \int_0^1 \log \left( 1+ \mu_{\bar{\omega}^k}(t-\alpha_i) \right)  dN^{i}_{t} \right]g(\alpha_{i}) d  \alpha_{i}
$$
and
$$
I_2 = I_2(N) =  \int_{\R^{n}}  \prod_{i=1}^n \exp\left[ -\int_0^1 \mu_{\omega}(t-\alpha_i)dt + \int_0^1 \log \left( 1+ \mu_{\omega}(t-\alpha_i) \right) dN^{i}_{t} \right] g(\alpha_{i}) d \alpha_{i},
$$
where to simplify the presentation of the proof we have taken $\rho(A) = 1$ i.e.\ $A=2$. Then, the following lemma holds (which is also valid for  $\rho(A) \neq 1$).

\begin{lemma}
\label{calcul} There exists  $0 < \delta < 1$ and a constant $p_0(\delta) > 0$ such that for any $k \in \{0 \dots 2^{D_{n}}-1\}$, any $\omega\in \lbrace 0,1 \rbrace^{2^{D_{n}}}$ and all sufficiently large $n$ 
$$
\PP_{ \lambda_{\omega} } \left( \mathcal{Q}_{k,\omega}(N) \geq \delta \right) \geq p_0(\delta)>0.
$$

\end{lemma}

\begin{proof}
For a function $\lambda \in L^{2}([0,1])$, we denote by $\| \lambda \| = \left( \int_{0}^{1} |\lambda(t)|^{2} dt \right)^{1/2}$ its $L_2$ norm and by  $\| \lambda \|_{\infty} = \sup_{t \in [0,1]} \left\{ |\lambda(t)|\right\}$ its supremum norm. In the proof, we repeatedly use the following inequalities that hold for any $\omega \in  \lbrace 0,1 \rbrace^{2^{D_{n}}}$
\begin{eqnarray}
 \|\mu_{ \omega }\| \leq & \|\mu_{ \omega }\|_{\infty} & \leq 2 c(\psi) m_{D_{n}}   \to 0, \label{eq:ineqmu}  \\ 
\|\lambda_{ \omega }\| \leq & \|\lambda_{ \omega }\|_{\infty}  & \leq \rho(A) + 2 c(\psi) m_{D_{n}} \to \rho(A)=1/2,  \nonumber 
\end{eqnarray}
as $n \to + \infty$.
Since for any $k$, one has $\int_{0}^1 \psi_{D,k}(t) dt = 0$, it follows that for
any $\omega$ and  $\alpha$,  $\int_0^1 \mu_{\omega}(t-\alpha)dt =   c(\psi) \xi_{D_{n}} 2^{D_{n}/2} =  c(\psi) m_{D_{n}}.$
Therefore, 
$$ I_1 = \int_{\R^{n}} g(\alpha_1)\dots g(\alpha_n)e^{-c(\psi) n m_{D_{n}} } \prod_{i=1}^n \exp\left[ \int_0^1 \log \left( 1+ \mu_{\bar{\omega}^{k}}(t-\alpha_i) \right) d N_{t}^{i} \right]d\alpha,$$
and
$$ I_2 = \int_{\R^{n}} g(\alpha_1)\dots g(\alpha_n)e^{-c(\psi) n m_{D_{n}} } \prod_{i=1}^n \exp\left[ \int_0^1 \log \left( 1+ \mu_{ \omega }(t-\alpha_i) \right) d N_{t}^{i} \right]d\alpha.$$
Let $z > 0$ be a positive real, and consider the following second order expansion of the logarithm
\begin{equation} \label{eq:log}
\log(1 + z) = z - \frac{z^2}{2}  +  \frac{z^3}{3} u^{-3} \mbox{ for some } 1 \leq u \leq 1 + z.
\end{equation}
Applying \eqref{eq:log}   implies that
\begin{equation} \label{eq:upboundlog}
\int_0^1 \log \left( 1+ \mu_{\bar{\omega}^{k}}(t-\alpha_i) \right) d N_{t}^{i}  \leq  \int_0^1  \left\{ \mu_{\bar{\omega}^{k}}(t-\alpha_i) - \frac{\mu^2_{\bar{\omega}^{k}}(t-\alpha_i)}{2}   \right\} d N_{t}^{i} + \int_0^1 \mu^3_{\bar{\omega}^{k}}(t-\alpha_i) d N_{t}^{i}, 
\end{equation}
and that
\begin{equation} \label{eq:lowboundlog}
\int_0^1 \log \left( 1+ \mu_{\omega}(t-\alpha_i) \right) d N_{t}^{i}  \geq  \int_0^1 \left\{ \mu_{\omega}(t-\alpha_i) - \frac{\mu^2_{\omega}(t-\alpha_i)}{2}   \right\} d N_{t}^{i}.
\end{equation}
Then, remark that inequalities \eqref{eq:ineqmu} imply that
\begin{eqnarray*}
\E_{\lambda_{ \omega }} \int_0^1 \mu^3_{\bar{\omega}^{k}}(t-\alpha_i) d N_{t}^{i} & = & \int_0^1 \mu^3_{\bar{\omega}^{k}}(t-\alpha_i) \int_{\R} \lambda_{\omega}(t-\tau_{i}) g(\tau_{i}) d \tau_{i}dt \\
& \leq  & \| \mu_{\bar{\omega}^{k}} \|_{\infty} \| \mu_{\bar{\omega}^{k}} \|^2  \| \lambda_{\omega}  \|_{\infty} = \mathcal{O}\left( m_{D_{n}}^3 \right)
\end{eqnarray*}
Therefore, by Markov's inequality it follows that $ \int_0^1 \mu^3_{\bar{\omega}^{k}}(t-\alpha_i) d N_{t}^{i} =  \mathcal{O}_{p}\left( m_{D_{n}}^3 \right)$ as $n \to +\infty$.
Hence, using inequality \eqref{eq:upboundlog}, one obtains that
\begin{eqnarray*}
I_2 & \leq &  e^{-c(\psi) n m_{D_{n}} + \mathcal{O}_{p}\left( n m_{D_{n}}^3 \right)}  \int_{\R^{n}} g(\alpha_1)\dots g(\alpha_n) \prod_{i=1}^n \exp\left[ \int_0^1 \left\{ \mu_{\omega}(t-\alpha_i) - \frac{\mu_{\omega}^2(t-\alpha_i)}{2}   \right\} dN_t^{i} \right ] d\alpha.
\end{eqnarray*}
and by inequality \eqref{eq:lowboundlog} it follows that
\begin{eqnarray*}
I_1 & \geq & e^{-c(\psi) n m_{D_{n}} } \int_{\R^{n}} g(\alpha_1)\dots g(\alpha_n) \prod_{i=1}^n \exp\left[ \int_0^1 \left\{ \mu_{\bar{\omega}^{k}}(t-\alpha_i) - \frac{\mu_{\bar{\omega}^{k}}^2(t-\alpha_i)}{2}  \right\} dN_t^{i} \right ] d\alpha,
\end{eqnarray*}
Combining the above inequalities and the Fubini's relation we obtain that
\begin{eqnarray}
\mathcal{Q}_{k,\omega}(N) & \geq  & e^{\mathcal{O}_p(n m_{D_{n}}^3)} \frac{ \prod_{i=1}^n \int_{\R} g(\alpha_i) \exp\left[ \int_0^1 \left\{\mu_{\bar{\omega}^{k}}(t-\alpha_i) - \frac{\mu_{\bar{\omega}^{k}}^2(t-\alpha_i)}{2} \right\} dN^i_t  \right] d\alpha_i} {\prod_{i=1}^n \int_{\R} g(\alpha_i) \exp\left[ \int_0^1  \left\{ \mu_{\omega}(t-\alpha_i)  - \frac{\mu_{\omega}^2(t-\alpha_i)}{2} \right\} d N^i_{t} \right] d\alpha_i}, \nonumber \\
& := & e^{\mathcal{O}_p(n m_{D_{n}}^3)} \frac{J_1}{J_2}.
\label{eq:I1I2} 
\end{eqnarray}
Let $z \in \R$  and consider the following second order expansion of the exponential
\begin{equation} \label{eq:exp}
\exp(z) = 1 +  z +  \frac{z^2}{2}  +  \frac{z^3}{6} \exp(u) \mbox{ for some } - |z| \leq u \leq  |z|.
\end{equation}
Let us now use  \eqref{eq:exp} with $z = \int_0^1 \left\{ \mu_{\bar{\omega}^{k}}(t-\alpha_i) -\frac{1}{2} \mu_{\bar{\omega}^{k}}^2(t-\alpha_i) \right\} dN^i_{t}$. By inequalities \eqref{eq:ineqmu}, one has that
\begin{eqnarray*}
\E_{\lambda_{ \omega }} | z | & \leq & \int_0^1 \left( \mu_{\bar{\omega}^{k}}(t-\alpha_i)  + \frac{1}{2} \mu_{\bar{\omega}^{k}}^2(t-\alpha_i) \right) \int_{\R} \lambda_{\omega}(t-\tau_{i}) g(\tau_{i}) d \tau_{i}dt, \\
& \leq & \|  \lambda_{\omega}\|_{\infty} \left( \| \mu_{\bar{\omega}^{k}} \| + \frac{1}{2} \| \mu_{\bar{\omega}^{k}} \|^2 \right) = \mathcal{O} \left( m_{D_{n}} \right).
\end{eqnarray*}
Therefore, $| z | =  \mathcal{O}_{p} \left( m_{D_{n}} \right) $ by Markov's inequality. 
Since $m_{D_{n}} \to 0$,  we obtain by using \eqref{eq:exp} that for each $i\in \lbrace 1,\dots, n \rbrace$, 
\begin{eqnarray*}
\lefteqn{\exp\left[ \int_0^1 \left\{ \mu_{\bar{\omega}^{k}}(t-\alpha_i) -\frac{1}{2} \mu_{\bar{\omega}^{k}}^2(t-\alpha_i) \right\} dN^i_{t} \right] } \\
& = & 1 + \int_0^1 \mu_{\bar{\omega}^{k}}(t-\alpha_i) dN^i_{t} - \frac{1}{2}\int_0^1 \mu_{\bar{\omega}^{k}}^2(t-\alpha_i) dN^i_{t}\\
&  &  + \frac{1}{2}\left( \int_0^1 \mu_{\bar{\omega}^{k}}(t-\alpha_i) dN^i_{t} - \frac{1}{2}\int_0^1 \mu_{\bar{\omega}^{k}}^2(t-\alpha_i) dN^i_{t} \right)^2 + \mathcal{O}_p(m_{D_{n}}^3),\\
& = &  1 + \int_0^1 \mu_{\bar{\omega}^{k}}(t-\alpha_i) dN^i_{t} - \frac{1}{2}\int_0^1 \mu_{\bar{\omega}^{k}}^2(t-\alpha_i) dN^i_{t}
+ \frac{1}{2}\left( \int_0^1 \mu_{\bar{\omega}^{k}}(t-\alpha_i) dN^i_{t} \right)^2+ \mathcal{O}_p(m_{D_{n}}^3),
\end{eqnarray*}
where we have used the fact that $$\left( \int_0^1 \mu_{\bar{\omega}^{k}}(t-\alpha_i) dN^i_{t} - \frac{1}{2}\int_0^1 \mu_{\bar{\omega}^{k}}^2(t-\alpha_i) dN^i_{t} \right)^2 =  \left( \int_0^1 \mu_{\bar{\omega}^{k}}(t-\alpha_i) dN^i_{t} \right)^2+ \mathcal{O}_p(m_{D_{n}}^3)$$.
From the definition of $J_1$ in (\ref{eq:I1I2}), we can use a stochastic version of the Fubini theorem (see \cite{Jacod}, Theorem 5.44) to obtain
\begin{eqnarray*}
J_1
& = & \prod_{i=1}^n  \left[ 1 +  \int_{\R} \int_0^1 g(\alpha_i) \mu_{\bar{\omega}^{k}}(t-\alpha_i)  dN^i_{t} d\alpha_i - \frac{1}{2}\int_{\R} \int_0^1 g(\alpha_i) \mu_{\bar{\omega}^{k}}^2(t-\alpha_i)  dN^i_{t} d\alpha_i \right.\\
& & \left. + \frac{1}{2} \int_{\R} g(\alpha_i) \left( \int_0^1 \mu_{\bar{\omega}^{k}}(t-\alpha_i)  dN^i_{t} \right)^2 d\alpha_i + \mathcal{O}_p(m_{D_{n}}^3) \right],\\
& = & \prod_{i=1}^n  \left[ 1 +  \int_0^1 g\star\mu_{\bar{\omega}^{k}}(t) dN^i_{t} - \frac{1}{2}\int_0^1  g\star\mu_{\bar{\omega}^{k}}^2(t) dN^i_{t} \right. \\ 
& &+\frac{1}{2} \int_{\R} g(\alpha_i) \left( \int_0^1 \mu_{\bar{\omega}^{k}}(t-\alpha_i) dN^i_{t} \right)^2 d\alpha_i \left. + \mathcal{O}_p(m_{D_{n}}^3) \right].
\end{eqnarray*}
At this step, it will be more convenient to work with the logarithm of the term $J_1$. We have
\begin{eqnarray*}
\ln(J_1)
& = & \sum_{i=1}^n  \ln \left[ 1 +  \int_0^1 g\star\mu_{\bar{\omega}^{k}}(t) dN^i_{t} - \frac{1}{2}\int_0^1  g\star\mu_{\bar{\omega}^{k}}^2(t) dN^i_{t} \right. \\ &  & \left. + \frac{1}{2} \int_{\R} g(\alpha_i) \left( \int_0^1 \mu_{\bar{\omega}^{k}}(t-\alpha_i) dN^i_{t} \right)^2 d\alpha_i + \mathcal{O}_p(m_{D_{n}}^3) \right].
\end{eqnarray*}
Using again the second order expansion of the logarithm \eqref{eq:log}, we obtain that
\begin{eqnarray*}
\ln(J_1)
& = & \sum_{i=1}^n \left[ \int_0^1 g\star\mu_{\bar{\omega}^{k}}(t) dN^i_{t} - \frac{1}{2}\int_0^1  g\star\mu_{\bar{\omega}^{k}}^2(t)  dN^i_{t} \right. \\
& & \hspace{-1cm} \left.+ \frac{1}{2} \int_{\R} g(\alpha_i) \left( \int_0^1 \mu_{\bar{\omega}^{k}}(t-\alpha_i) dN^i_{t}  \right)^2 d\alpha_i - \frac{1}{2} \left( \int_0^1 g\star\mu_{\bar{\omega}^{k}}(t)  dN^i_{t} \right )^2 + \mathcal{O}_p(m_{D_{n}}^3) \right].
\end{eqnarray*}
Using similar arguments for the term $J_2$ defined in (\ref{eq:I1I2}), we obtain that
\begin{eqnarray*}
\ln(J_2)
& = & \sum_{i=1}^n \left[ \int_0^1 g\star\mu_{\omega}(t) dN^i_{t} - \frac{1}{2}\int_0^1  g\star\mu_{\omega}^2(t)  dN^i_{t} \right. \\
& & \hspace{-1cm} \left.+ \frac{1}{2} \int_{\R} g(\alpha_i) \left( \int_0^1 \mu_{\omega}(t-\alpha_i) dN^i_{t}  \right)^2 d\alpha_i - \frac{1}{2} \left( \int_0^1 g\star\mu_{\omega}(t)  dN^i_{t} \right )^2 + \mathcal{O}_p(m_{D_{n}}^3) \right].
\end{eqnarray*}
Combing the above equalities for $J_{1}$ and $J_{2}$, we obtain the following lower bound  for $\ln(\mathcal{Q}_{k,\omega}(N))$
\begin{eqnarray}
 \hspace{-1cm}  \ln(\mathcal{Q}_{k,\omega}(N)) & \geq & \ln(J_1) - \ln(J_2)+  \mathcal{O}_p(n m_{D_{n}}^3)  \nonumber \\
 & = &  \sum_{i=1}^n \big\{ \mathbb{E}_{\lambda_{\omega}}\left(\int_0^1 g\star\lbrace \mu_{\bar{\omega}^k}(t) - \mu_{\omega}(t)\rbrace dN^i_{t}\right) + \frac{1}{2} \| g\star\lambda_{\omega} \|^2 - \frac{1}{2} \| g\star\lambda_{\bar{\omega}^k} \|^2 \label{eq:bourin2} \\
 &  & +   \int_0^1 g\star\lbrace \mu_{\bar{\omega}^k}(t) - \mu_{\omega}(t) \rbrace dN^i_{t} - \mathbb{E}_{\lambda_{\omega}}\left(\int_0^1 g\star\lbrace \mu_{\bar{\omega}^k}(t) - \mu_{\omega}(t) \rbrace dN^i_{t}\right)  \label{eq:bourin1} \\
& &  + \frac{1}{2}\int_0^1  g\star\mu_{\omega}^2(t) dN^i_{t} - \frac{1}{2} \int_\R g(\alpha_i) \left( \int_0^1 \mu_{\omega}(t-\alpha_i) dN^i_{t} \right)^2 d\alpha_i  \label{eq:bourin3} \\
& & + \frac{1}{2} \int_\R g(\alpha_i) \left( \int_0^1 \mu_{\bar{\omega}^k}(t-\alpha_i) dN^i_{t} \right)^2 d\alpha_i - \frac{1}{2}\int_0^1  g\star\mu_{\bar{\omega}^k}^2(t) dN^i_{t}  \label{eq:bourin4} \\
& & \hspace{-1.5cm} - \frac{1}{2} \left( \int_0^1 g\star\mu_{\bar{\omega}^k}(t) dN^i_{t} \right )^2 + \frac{1}{2} \| g\star\lambda_{\bar{\omega}^k} \|^2 +\frac{1}{2} \left( \int_0^1 g\star\mu_{\omega}(t) dN^i_{t} \right )^2  -\frac{1}{2} \| g\star\lambda_{\omega} \|^2   \big\} \label{eq:bourin5} \\
& &  +\mathcal{O}_p(n m_{D_{n}}^3) \nonumber. 
\end{eqnarray}
In what follows, we will show that, for all sufficiently large $n$, the terms (\ref{eq:bourin2})-(\ref{eq:bourin5}) are bounded from below (in probability). Since $n m_{D_{n}}^3 \to 0$, this will imply that there exists  $c > 0$ (not depending on $\lambda_{\omega}$) and a constant $p(c) > 0$  such that for all sufficiently large $n$
$$
\PP_{\lambda_{\omega} }  \left( \ln \left( \mathcal{Q}_{k,\omega}(N) \right) \geq - c \right) = \PP_{\lambda_{\omega} }  \left( \mathcal{Q}_{k,\omega}(N) \geq \exp(-c) \right) \geq p(c)  > 0
$$
which is the result stated in Lemma \ref{calcul}. \\

\noindent \underline{Lower bound for (\ref{eq:bourin2})}: since for any $1 \leq i \leq n$
\begin{eqnarray*}
\mathbb{E}_{\lambda_{\omega} } \left(\int_0^1 g\star\lbrace \mu_{\bar{\omega}^k}(t) - \mu_{\omega}(t) \rbrace  dN^i_{t}\right) & = & \int_0^1 g\star\lbrace  \lambda_{\bar{\omega}^k}(t) - \lambda_{\omega}(t)\rbrace  \{g\star\lambda_{\omega}(t) \} dt.
\end{eqnarray*}
We obtain that 
\begin{eqnarray*}
 \sum_{i=1}^n \left[ \mathbb{E}_{\lambda_{\omega}}\left(\int_0^1 g\star\lbrace  \lambda_{\bar{\omega}^k}(t) - \lambda_{\omega}(t)\rbrace dN^i_{t} + \frac{1}{2} \| g\star\lambda_{\omega} \|^2 - \frac{1}{2} \| g\star\lambda_{\bar{\omega}^k} \|^2 \right) \right]  & = &  - \frac{n}{2} \| g\star\lbrace \mu_{\omega} - \mu_{\bar{\omega}^k} \rbrace \|^2 
\end{eqnarray*}
Remark that $\mu_{\omega} - \mu_{\bar{\omega}^k} = \pm \xi_D \psi_{Dk}$.  In what follows we will repeatidely use the following relation
 \begin{equation} \label{eq:psistarg}
\| \psi_{Dk}\star g \|^2 = \int_0^1 \left( \psi_{Dk}\star g(t) \right)^2dt = \sum_{\ell\in \Omega_D} | c_{\ell}( \psi_{Dk} ) |^2 |\gamma_{\ell}|^2 \asymp 2^{-2 D \nu}
\end{equation}
which follows from Assumption \ref{ass:gdecay} combined with Parseval's relation,  from the fact that $\# \Omega_D \asymp 2^D$ and that under Assumption \ref{ass:g} $|\gamma_{\ell}| \asymp 2^{-D \nu}$ for all $\ell \in \Omega_D$. Therefore
\begin{eqnarray*}
\| g\star \lbrace  \mu_{\omega} - \mu_{\bar{\omega}^k} \rbrace \|^2 = \xi_D^2 \int_0^1 \left( \psi_{Dk}\star g(t) \right)^2dt  \asymp \xi_D^2 2^{-2 D \nu} \asymp   n^{-1}. 
\end{eqnarray*}
Therefore, 
$$
- \sum_{i=1}^n \left[ \mathbb{E}_{\lambda_{\omega}}\left(\int_0^1 g\star \lbrace  \lambda_{\bar{\omega}^k}(t) - \lambda_{\omega}(t)\rbrace dN^i_{t} + \frac{1}{2} \| g\star \lambda_{\omega} \|^2 - \frac{1}{2} \| g\star \lambda_{\bar{\omega}^k} \|^2 \right) \right] \asymp 1,
$$
which implies that there exists a constant $0 < c_{0} < + \infty $ such that for all sufficiently large $n$ the deterministic term (\ref{eq:bourin2}) satisfies
$$
(\ref{eq:bourin2})  = \sum_{i=1}^n \left[ \mathbb{E}_{\lambda_{\omega}}\left(\int_0^1 g\star \lbrace  \lambda_{\bar{\omega}^k}(t) - \lambda_{\omega}(t)\rbrace dN^i_{t} + \frac{1}{2} \| g\star \lambda_{\omega} \|^2 - \frac{1}{2} \| g\star \lambda_{\bar{\omega}^k} \|^2 \right) \right]   \geq -c_{0}. 
$$ 

In the rest of the proof, we show that, for all sufficiently large $n$, the terms (\ref{eq:bourin1})-(\ref{eq:bourin5}) are bounded from below in probability. Without loss of generality, we consider only in what follows the case  $\mu_{\omega} - \mu_{\bar{\omega}^k} =  \xi_D \psi_{Dk}$.  \\

\noindent \underline{Lower bound for (\ref{eq:bourin1})}: rewrite first (\ref{eq:bourin1}) as
$$
(\ref{eq:bourin1}) =    -   \xi_D  \sum_{i=1}^n \int_0^1  g\star  \psi_{D,k}(t) d\tilde{N}^i_{t}, 
$$
where $d\tilde{N}^i_{t} = d\tilde{N}^i_{t} - \lambda(t-\tau_{i})dt$.
Then, using the fact that, conditonnaly to $\tau_{1},\ldots,\tau_{n}$, the counting process $\sum_{i=1}^n N^i $ is a Poisson process with intensity $\sum_{i=1}^{n} \lambda_{\omega}(t-\tau_{i})$, it follows from an analogue of Bennett's inequality for Poisson processes (see e.g.\   Proposition 7 in \cite{Reynaud}) that for any $y > 0$


\begin{eqnarray*}
\PP \left( \left| \xi_D  \sum_{i=1}^n \int_0^1  g\star  \psi_{D,k}(t) d\tilde{N}^i_{t} \right| \right. & \leq  & \sqrt{2 y \xi^2_D \int_{0}^{1} \sum_{i=1}^{n} | g\star  \psi_{D,k}(t) |^2 \lambda_{\omega}(t-\tau_{i})dt}  \\
& & \left.  +  \frac{1}{3} y \xi_D  \|  g\star  \psi_{D,k} \|_{\infty} \big| \tau_{1},\ldots,\tau_{n} \right) \geq 1-\exp\left(-y \right)
\end{eqnarray*}
Since $  \int_{0}^{1} \sum_{i=1}^{n} | g\star  \psi_{D,k}(t) |^2 \lambda_{\omega}(t-\tau_{i})dt \leq   n  \| g\star  \psi_{D,k} \|^{2} \| \lambda_{\omega} \|_{\infty} $ for any $\tau_{1},\ldots,\tau_{n}$, it follows that for  $y = \log(2)$
\begin{eqnarray*}
\PP \left( \left| \xi_D  \sum_{i=1}^n \int_0^1  g\star  \psi_{D,k}(t) d\tilde{N}^i_{t} \right| \right. & \leq  & \sqrt{2 \log(2) \xi^2_D  n  \| g\star  \psi_{D,k}(t) \|^{2} \| \lambda_{\omega} \|_{\infty} }    \left.  +  \frac{1}{3} \log(2) \xi_D  \|  g\star  \psi_{D,k} \|_{\infty}  \right) \geq 1/2.
\end{eqnarray*}
Now, using that $\xi^2_D  n  \| g\star  \psi_{D,k}(t) \|^{2}  \| \lambda_{\omega} \|_{\infty}  \asymp 1$ and $\xi_D  \|  g\star  \psi_{D,k} \|_{\infty} \leq  \| \psi \|_{\infty} 2^{D/2} \xi_D \to 0$, it follows that there exists a constant $c_{1} > 0$ such that for all sufficiently large $n$
\begin{equation}
 \PP \left(  \left| (\ref{eq:bourin1})    \right|  \leq   c_{1}  \right) = \PP \left( \left| \xi_D  \sum_{i=1}^n \int_0^1  g\star  \psi_{D,k}(t) d\tilde{N}^i_{t} \right|  \leq   c_{1}  \right) \geq 1/2.
\end{equation}

\noindent \underline{Lower bound for (\ref{eq:bourin3}) and (\ref{eq:bourin4})}: define
\begin{eqnarray*}
X_{i} & = &  \frac{1}{2}\int_0^1  g\star \mu_{\omega}^2(t) dN^i_{t} - \frac{1}{2}\int_0^1  g\star \mu_{\bar{\omega}^k}^2(t) dN^i_{t} \\
& & + \frac{1}{2} \int_\R g(\alpha_i) \left( \int_0^1 \mu_{\bar{\omega}^k}(t-\alpha_i) dN^i_{t} \right)^2 d\alpha_i - \frac{1}{2} \int_\R g(\alpha_i) \left( \int_0^1 \mu_{\omega}(t-\alpha_i) dN^i_{t} \right)^2 d\alpha_i,
\end{eqnarray*}
and note that $(\ref{eq:bourin3}) + (\ref{eq:bourin4}) = \sum_{i=1}^{n} X_{i}$.
For any $1 \leq i \leq n$
\begin{eqnarray*}
\E_{\lambda_{\omega}} X_{i} & = & \frac{1}{2} \int_\R g(\alpha_i) \left(  \left(  \int_0^1 \mu_{\bar{\omega}^k}(t-\alpha_i) g \star  \lambda_{\omega}(t) dt \right)^2 -  \left(  \int_0^1 \mu_{\omega}(t-\alpha_i) g \star  \lambda_{\omega}(t) dt \right)^2\right) d\alpha_i \\
& = &  \frac{1}{2} \int_\R g(\alpha_i) \left(    \left( \int_0^1  - \xi_{D} \psi_{D,k}(t-\alpha_i)  g \star  \lambda_{\omega}(t) dt \right)  \left(  \int_0^1 \left( \mu_{\omega}(t-\alpha_i) - \mu_{\bar{\omega}^k}(t-\alpha_i) \right) g \star  \lambda_{\omega}(t) dt \right)   \right) d\alpha_i \\
&  = &  \frac{1}{2} \int_\R g(\alpha_i) \left(    \left( \int_0^1  - \xi_{D} \psi_{D,k}(t-\alpha_i)  g \star  \mu_{\omega}(t) dt \right)   \left(  \int_0^1 \left( \mu_{\omega}(t-\alpha_i) - \mu_{\bar{\omega}^k}(t-\alpha_i) \right) g \star  \mu_{\omega}(t) dt \right)   \right) d\alpha_i
\end{eqnarray*}
which implies that
$$
\left| \E_{\lambda_{\omega}} X_{i} \right|  \leq  \frac{1}{2}   \xi_{D} 2^{D/2} \|  \psi \|_{\infty}  \|  \mu_{\omega} \|_{\infty}^2\left( \|  \mu_{\omega} \|_{\infty} + \|  \mu_{\bar{\omega}^k} \|_{\infty} \right) \asymp m_{D}^4.
$$
Therefore
$
\sum_{i=1}^{n} \E_{\lambda_{\omega}} X_{i} \to 0 \mbox{ as } n \to + \infty,
$
since $n m_{D}^4 \to 0 $.
Now, remark that $X_{1},\ldots,X_{n}$ are i.i.d variables satisfying for all $1 \leq i \leq n$
\begin{equation} \label{eq:boundXi}
|X_{i}| \leq \frac{1}{2} (\| \mu_{\omega} \|_{\infty}^2 + \| \mu_{\bar{\omega}^k} \|_{\infty}^2 )(K_{i} + K_{i}^2) \leq  2 c^2(\psi) m^2_{D_{n}}  (K_{i} + K_{i}^2)
\end{equation}
where $K_{i} =  \int_0^1  dN^i_{t}$. Conditionally to $\tau_{i}$, $K_{i}$ is a Poisson variable with intensity $\int_{0}^{1} \lambda_{\omega}(t-\tau_{i}) dt = \int_{0}^{1} \lambda_{\omega}(t)dt = \| \lambda_{\omega} \|_{1} $. Hence, the bound \eqref{eq:ineqmu}  for $ \| \lambda_{\omega} \|_{\infty}$ and inequality \eqref{eq:boundXi} implies that there exists a constant $C > 0$ (not depending on $\lambda_{\omega}$)  such that
$$
 \E X_{1}^{2} \leq C m^4_{D_{n}},
$$
which implies that $\var( \sum_{i=1}^{n} X_{i})  = n \var(X_{1}) \leq n  \E X_{1}^{2}  \to 0$ as $n \to + \infty $ since $n m_{D}^4 \to 0 $. Therefore, $(\ref{eq:bourin3}) + (\ref{eq:bourin4}) =\sum_{i=1}^{n} X_{i}$ converges to zero in probability as $n \to + \infty $ using Chebyshev's inequality. \\

\noindent \underline{Lower bound for (\ref{eq:bourin5})}:
we denote by $S_i$ the difference 
$$
S_i := 2 \left( - \frac{1}{2} \left( \int_0^1 g\star \mu_{\bar{\omega}^k}(t) dN^i_{t} \right )^2 + \frac{1}{2} \| g\star \lambda_{\bar{\omega}^k} \|^2 +\frac{1}{2} \left( \int_0^1 g\star \mu_{\omega}(t) dN^i_{t} \right )^2 -\frac{1}{2} \| g\star \lambda_{\omega} \|^2 \right),
$$
and remark that $(\ref{eq:bourin5}) = \frac{1}{2} \sum_{i=1}^{n} S_i $.
First, remark that 
\begin{eqnarray*}
\mathbb{E}_{\lambda_{\omega}}  S_i& =  &\| g\star \lambda_{\bar{\omega}^k} \|^2-\| g\star \lambda_{\omega} \|^2 + \int_{0}^1 (g\star \mu_{\omega})^2(t) g\star \lambda_{\omega}(t ) dt - \int_{0}^1 (g\star \mu_{\bar{\omega}^k})^2(t) g\star  \lambda_{\omega}(t) dt  \\
& & + \int_{\R} g(\tau_{i}) \left( \left\{\int_{0}^1 (g\star \mu_{\omega})(t) \lambda_{\omega}(t-\tau_i) dt \right\}^2 - \left\{\int_{0}^1 (g\star \mu_{\bar{\omega}^k})(t) \lambda_{\omega}(t-\tau_i) dt \right\}^2 \right) d\tau_{i}.
\end{eqnarray*}
Since
$
\| g\star \mu_{\bar{\omega}^k} \|^2-\| g\star \mu_{\omega} \|^2 = \| g\star \lambda_{\bar{\omega}^k} \|^2-\| g\star \lambda_{\omega} \|^2
$
and
$
g\star \lambda_{\omega} = 1 + g\star \mu_{\omega}
$
it follows that
\begin{eqnarray*}
\mathbb{E}_{\lambda_{\omega}}  S_i& =  &  \underbrace{  \int_{0}^1(g\star \mu_{\omega})^2(t) g\star \mu_{\omega}(t) dt - \int_{0}^1(g\star \mu_{\bar{\omega}^k})^2(t) g\star \mu_{\omega}(t) dt}_{S_{i,1}}\\
 & & +\underbrace{  \int_{\R} g(\tau_{i}) \left(  \left\{\int_{0}^1 (g\star \mu_{\omega})(t) \lambda_{\omega}(t-\tau_i) dt \right\}^2 - \left\{\int_{0}^1 (g\star \mu_{\bar{\omega}^k})(t) \lambda_{\omega}(t-\tau_i) dt \right\}^2 \right) d\tau_{i} }_{S_{i,2}}.
\end{eqnarray*}
One has that
$$
|S_{i,1} | \leq \| \mu_{\omega} \|_{\infty}^3 + \| \mu_{\bar{\omega}^k} \|_{\infty}^2 \| \mu_{\omega} \|_{\infty}   \leq 16 c^3(\psi) m_{D_{n}}^3,
$$
and that
\begin{eqnarray*}
S_{i,2} & =  &   \xi_{D}^2  \int_{\R} g(\tau_{i}) \left(   \left( \int_{0}^1 g\star \psi_{D,k} (t) \lambda_{\omega}(t-\tau_i) dt \right) \left( \int_{0}^1 g\star (\mu_{\omega} + \mu_{\bar{\omega}^k})(t) \lambda_{\omega}(t-\tau_i) dt \right) \right)  d\tau_{i}
\end{eqnarray*}
Hence using \eqref{eq:ineqmu} and  \eqref{eq:psistarg}  it follows that there exists a constant $C > 0$ such that for all sufficiently large $n$
$$
|S_{i,2}| \leq \xi_{D}^2 \| g\star \psi_{D,k} \| \left( \| \mu_{\omega} \|_{\infty} +  \| \mu_{\bar{\omega}^k} \|_{\infty}   \right) \leq C n^{-\frac{3s + \nu +1}{2s + 2 \nu +1}}
$$
Then, since $ s > 2 \nu +1 > \nu$ it follows that
$$
\sum_{i=1}^{n} \mathbb{E}_{\lambda_{\omega}}  S_i = \mathcal{O}\left( n^{-\frac{(s - 2\nu -1)}{2s + 2 \nu +1}} + n^{-\frac{(s - \nu)}{2s + 2 \nu +1}} \right) \to 0.
$$
Now, note that $\var( \sum_{i=1}^{n}  S_i ) = n \var(Y_{1})$ where
$$
Y_{1} =     \left( \int_0^1 g\star \mu_{\omega}(t) dN^1_{t} \right )^2 - \left( \int_0^1 g\star \mu_{\bar{\omega}^k}(t) dN^1_{t} \right )^2 .
$$
Since $| Y_{1}  | \leq \left( \|   \mu_{\omega} \|^{2}_{\infty} +   \| \mu_{\bar{\omega}^k} \|^{2}_{\infty} \right) K_{1}^2$ with $K_{1} =  \int_0^1 dN^1_{t}$ being, conditionally to $\tau_{1}$, a Poisson variable with intensity $\int_{0}^{1} \lambda_{\omega}(t-\tau_{1}) dt = \int_{0}^{1} \lambda_{\omega}(t)dt = \| \lambda_{\omega} \|_{1} $. Therefore, using \eqref{eq:ineqmu}, it follows that there exists a constant $C > 0$ (not depending on $\lambda_{\omega}$)   such that
$$
\var( \sum_{i=1}^{n}  S_i ) = n \var(Y_{1}) \leq n \E Y_{1}^2 \leq C n  m_{D}^4 \to 0.
$$
Therefore,  using Chebyshev's inequality, we obtain that  $(\ref{eq:bourin5}) = \frac{1}{2} \sum_{i=1}^{n} S_i $ converges to zero in probability as $n \to + \infty $, which ends the proof of the lemma. \hfill $\square$ 


\end{proof}

\subsection{Lower bound on $B_{s}^{p,q}(A)$}

By applying inequality \eqref{eq:boundS} and Lemma \ref{calcul}, we obtain that there exists  $0 < \delta < 1$ such that  for all sufficiently large $n$
$$
\inf_{\hat \lambda_{n}} \sup_{\lambda \in S_D(A)} \mathbb{E}_{\lambda} \| \hat \lambda_{n} - \lambda \|^2_{2} \geq C \xi_{D_{n}}^2 2^{D_{n}},
$$
for some  constant $C > 0$ that is independent of $D_{n}$. From the definition \eqref{eq:fonctionstests} of $\xi_D$ and using the choice \eqref{eq:Dn} for $D_{n}$, we obtain that
$$
\inf_{\hat \lambda_{n}} \sup_{\lambda \in S_D(A)} \mathbb{E}_{\lambda} \| \hat \lambda_{n} - \lambda \|^2_{2} \geq  C \xi_{D_{n}}^2 2^{D_{n}} \asymp 2^{-2s D_{n}}  \asymp n^{-\frac{2s}{2s+2\nu+1}}.
$$
Now, since $S_{D}(A) \subset B_{p,q}^{s}(A)$  for any $D \geq 1$ we obtain from the above inequalities that there exists a constant $C_{0} > 0$ such that for all sufficiently large $n$
\begin{eqnarray*}
\inf_{\hat \lambda_{n}} \sup_{\lambda \in B_{p,q}^{s}(A)  \bigcap  \Lambda_{0}} n^{\frac{2s}{2s+2\nu+1}} \mathbb{E}_{\lambda} \| \hat \lambda_{n} - \lambda \|^2_{2}
& \geq & \inf_{\hat \lambda_{n}} \sup_{\lambda \in S_{D_{n}}(A)} \mathbb{E}_{\lambda} \| \hat \lambda - \lambda \|^2_{2},\\
& \geq &  C_{0} n^{-\frac{2s}{2s+2\nu+1}},
\end{eqnarray*}
which concludes the proof of Theorem \ref{th:borneinf}.\hfill $\square$

\bibliographystyle{plain}
\bibliography{poisson_shift}

\begin{thebibliography}{10}

\bibitem{MR2291495}
A.~Antoniadis and J.~Bigot.
\newblock Poisson inverse problems.
\newblock {\em Ann. Statist.}, 34(5):2132--2158, 2006.

\bibitem{MR2676894}
J.~Bigot and S.~Gadat.
\newblock A deconvolution approach to estimation of a common shape in a shifted
  curves model.
\newblock {\em Ann. Statist.}, 38(4):2422--2464, 2010.

\bibitem{MR2727451}
J.~Bigot, S.~Gadat, and C.~Marteau.
\newblock Sharp template estimation in a shifted curves model.
\newblock {\em Electron. J. Stat.}, 4:994--1021, 2010.

\bibitem{Bremaud}
P.~Br\'emaud.
\newblock {\em Point processes and Queues, Martingale Dynamics}.
\newblock Springer series in Statistics, 1981.

\bibitem{cavgopitsy}
L.~Cavalier, G.~Golubev, D.~Picard, and A.~Tsybakov.
\newblock Oracle inequalities for inverse problems.
\newblock {\em Ann. Statist.}, 30(3):843--874, 2002.
\newblock Dedicated to the memory of Lucien Le Cam.

\bibitem{MR1930348}
L.~Cavalier and J.-Y. Koo.
\newblock Poisson intensity estimation for tomographic data using a wavelet
  shrinkage approach.
\newblock {\em IEEE Trans. Inform. Theory}, 48(10):2794--2802, 2002.

\bibitem{MR1268002}
D.~L. Donoho.
\newblock Nonlinear wavelet methods for recovery of signals, densities, and
  spectra from indirect and noisy data.
\newblock In {\em Different perspectives on wavelets ({S}an {A}ntonio, {TX},
  1993)}, volume~47 of {\em Proc. Sympos. Appl. Math.}, pages 173--205. Amer.
  Math. Soc., Providence, RI, 1993.

\bibitem{MR2649451}
F.-X. Dup{\'e}, J.~M. Fadili, and J.-L. Starck.
\newblock A proximal iteration for deconvolving {P}oisson noisy images using
  sparse representations.
\newblock {\em IEEE Trans. Image Process.}, 18(2):310--321, 2009.

\bibitem{HKPT}
W.~Hardle, G.~Kerkyacharian, D.~Picard, and A.~Tsybakov.
\newblock {\em Wavelets, Approximation and. Statistical Applications.}, volume
  129.
\newblock Lecture Notes in Statistics, New York: Spriner-Verlag, 1998.

\bibitem{Jacod}
J.~Jacod.
\newblock {\em Calcul stochastique et probl\`emes de martingales}.
\newblock Springer Verlag, 1979.

\bibitem{johnstone}
I.~M. Johnstone.
\newblock Function estimation in gaussian noise: Sequence models.
\newblock {\em Unpublished Monograph, http://www-stat.stanford.edu/$\tilde{
  }$imj/}, 2002.

\bibitem{JKPR}
I.~M. Johnstone, G.~Kerkyacharian, D.~Picard, and M.~Raimondo.
\newblock Wavelet deconvolution in a periodic setting.
\newblock {\em J. Roy. Statist. Soc. Ser. B}, 66:547--573, 2004.

\bibitem{poisson-procs}
J.~F.~C. Kingman.
\newblock {\em Poisson Processes}, volume~3 of {\em Oxford Studies in
  Probability}.
\newblock Oxford University Press, Oxford, 1993.

\bibitem{MR1678884}
E.~D. Kolaczyk.
\newblock Wavelet shrinkage estimation of certain {P}oisson intensity signals
  using corrected thresholds.
\newblock {\em Statist. Sinica}, 9(1):119--135, 1999.

\bibitem{massart}
P.~Massart.
\newblock {\em Concentration Inequalities and Model Selection: Ecole d'été de
  Probabilités de Saint-Flour XXXIII - 2003}.
\newblock Lecture Notes in Mathematics Springer, 2006.

\bibitem{Meyer}
Y.~Meyer.
\newblock {\em Ondelettes et opérateurs, I}.
\newblock Hermann, 1989.

\bibitem{MR1790322}
R.~D. Nowak and E.~D. Kolaczyk.
\newblock A statistical multiscale framework for {P}oisson inverse problems.
\newblock {\em IEEE Trans. Inform. Theory}, 46(5):1811--1825, 2000.
\newblock Information-theoretic imaging.

\bibitem{MR2488345}
M.~Pensky and T.~Sapatinas.
\newblock Functional deconvolution in a periodic setting: uniform case.
\newblock {\em Ann. Statist.}, 37(1):73--104, 2009.

\bibitem{Reynaud}
P.~Reynaud-Bourret.
\newblock Adaptive estimation of the intensity of inhomogeneous poisson
  processes via concentration inequalities.
\newblock {\em Probability Theory and Related Fields}, 126:103--153, 2003.

\bibitem{Reynaud2}
P.~Reynaud-Bourret and V.~Rivoirard.
\newblock Near optimal thresholding estimation of a poisson intensity on the
  real line.
\newblock {\em Electronic journal of statistics}, 4:171--238, 2010.

\bibitem{MR0440354}
H.~P. Rosenthal.
\newblock On the span in {$L^{p}$} of sequences of independent random
  variables. {II}.
\newblock In {\em Proceedings of the {S}ixth {B}erkeley {S}ymposium on
  {M}athematical {S}tatistics and {P}robability ({U}niv. {C}alifornia,
  {B}erkeley, {C}alif., 1970/1971), {V}ol. {II}: {P}robability theory}, pages
  149--167, Berkeley, Calif., 1972. Univ. California Press.

\bibitem{bioman}
J.~Wang, A.~Huda, V.~Lunyak, and I.~Jordan.
\newblock A gibbs sampling strategy applied to the mapping of ambiguous
  short-sequence tags.
\newblock {\em Bioinformatics}, 26:2501--2508, 2010.

\bibitem{MR2417680}
R.~M. Willett and R.~D. Nowak.
\newblock Multiscale {P}oisson intensity and density estimation.
\newblock {\em IEEE Trans. Inform. Theory}, 53(9):3171--3187, 2007.

\bibitem{MR2516691}
B.~Zhang, J.~M. Fadili, and J.-L. Starck.
\newblock Wavelets, ridgelets, and curvelets for {P}oisson noise removal.
\newblock {\em IEEE Trans. Image Process.}, 17(7):1093--1108, 2008.

\bibitem{superman}
Y.~Zhang, T.~Liu, C.~Meyer, J.~Eeckhoute, D.~Johnson, B.~Bernstein,
  C.~Nussbaum, R.~Myers, M.~Brown, W~Li, and X.~Liu.
\newblock Model-based analysis of chip-seq (macs).
\newblock {\em Genome Biology}, 9, 2008.

\end{thebibliography}

\end{document}